\documentclass[english]{article} 
\usepackage{amsmath}
\usepackage{amssymb}
\usepackage{amsthm}
\usepackage[latin1]{inputenc}
\usepackage{xcolor}
\usepackage{adjustbox}
\usepackage{tikzscale}  
\usepackage{bussproofs}
\usepackage{soul}
\usepackage{dsfont}
\usepackage{hyperref}
\usepackage{pgf}
\usepackage{tikz}
\usepackage{multirow}
\usepackage{enumitem}   
\usepackage{pifont}      
\usepackage{multicol}     

\usetikzlibrary{arrows,shapes, positioning,automata,calc,fit}

\usepackage{pdflscape} 
\usepackage{afterpage} 

\usepackage{enumitem} 

\hypersetup{ colorlinks = false, linkbordercolor = white}

\title{Translations: generalizing relative expressiveness between logics}
\author{Diego Pinheiro Fernandes\footnote{PhD student at University of Salamanca, Spain}}

\date{\today}

\newcommand{\T}{\mathcal{T}}
\newcommand{\ibid}{\emph{ibid}}
\newcommand{\U}{\mathfrak{U}}

\newcommand{\M}{\mathcal{M}}
\newcommand{\F}{\mathcal{F}}
\newcommand{\R}{\mathcal{R}}
\newcommand{\WPL}{\mathcal{WPL}}

\newcommand{\expr}{\preccurlyeq}

\newcommand{\restr}{\mathord{\upharpoonright}}
\renewcommand{\L}{\mathcal{L}}

\newcommand{\s}{\hspace{6mm}}

\newcommand{\MSL}{\mathcal{MSL}}
\newcommand{\CL}{\mathcal{CL}}
\newcommand{\IL}{\mathcal{IL}}

\newcommand{\IPL}{\mathcal{IPL}}
\newcommand{\CPL}{\mathcal{CPL}}
\newcommand{\FOL}{\mathcal{FOL}}

\usepackage{turnstile}
\usepackage{adjustbox}

\renewcommand{\makehor}[4]
  {\ifthenelse{\equal{#1}{n}}{\hspace{#3}}{}
   \ifthenelse{\equal{#1}{s}}{\rule[-0.5#2]{#3}{#2}}{}
   \ifthenelse{\equal{#1}{d}}{\setlength{\lengthvar}{#2}
     \addtolength{\lengthvar}{0.5#4}
     \rule[-\lengthvar]{#3}{#2}
     \hspace{-#3}
     \rule[0.5#4]{#3}{#2}}{}
   \ifthenelse{\equal{#1}{t}}{\setlength{\lengthvar}{1.5#2}
     \addtolength{\lengthvar}{#4}
     \rule[-\lengthvar]{#3}{#2}
     \hspace{-#3}
     \rule[-0.5#2]{#3}{#2}
     \hspace{-#3}
     \setlength{\lengthvar}{0.5#2}
     \addtolength{\lengthvar}{#4}
     \rule[\lengthvar]{#3}{#2}}{}
   \ifthenelse{\equal{#1}{w}}{
     \setbox0=\hbox{$\sim$}%
     \raisebox{-.6ex}{\hspace*{-.05ex}\adjustbox{width=#3,height=\height}{\clipbox{0.75 0 0 0}{\usebox0}}}}{}
  }

\usepackage{mathrsfs,stackengine}

\newtheorem{definition}[subsection]{Definition}
\newtheorem{proposition}[subsection]{Proposition}

\newtheorem{theorem}[subsection]{Theorem}
\newtheorem{corollary}[subsection]{Corollary}

\usepackage{datetime}  
\usepackage{fancyhdr}
\usepackage{currfile}
\begin{document}

\maketitle

\begin{abstract}
	There is a strong demand for precise means for the comparison of logics in terms of expressiveness both from theoretical and from application areas. The aim of this paper is to propose a sufficiently general and reasonable formal criterion for expressiveness, so as to apply not only to model-theoretic logics, but also to Tarskian and proof-theoretic logics. For model-theoretic logics there is a standard framework of relative expressiveness, based on the capacity of characterizing structures, and a straightforward formal criterion issuing from it. The problem is that it only allows the comparison of those logics defined within the same class of models. The urge for a broader framework of expressiveness is not new. Nevertheless, the enterprise is complex and a reasonable model-theoretic formal criterion is still wanting. Recently there appeared two criteria in this wider framework, one from García-Matos \& Väänänen and other from L. Kuijer. We argue that they are not adequate. Their limitations are analysed and we propose to move to an even broader framework lacking model-theoretic notions, which we call ``translational expressiveness''. There is already a criterion in this later framework by Mossakowski et al., however it turned out to be too lax. We propose some adequacy criteria for expressiveness and a formal criterion of translational expressiveness complying with them is given.

\end{abstract}

\tableofcontents										

\section{Introduction}

It is very common for those who work with logic to make comparisons such as ``the logic $\L'$ is more expressive than $\L$'', ``$\L'$ is stronger than $\L$'', ``$\L$ is included in $\L'$'', ``$\L$ can be reduced to $\L'$'', etc. Such assertions are often made on imprecise grounds and, though possibly being non-ambiguous and non-problematic, the lack of clarity around the usage of these concepts can generate terminological confusion across the literature (e.g. \cite{Humberstone-BTP}) and harden the comparison of formal results.

In the literature, the notion of logic inclusion or sub-logic (these terms will be used interchangeably here) is pretty much linked with language and axiomatic extensions, which on their turn are linked with ``strength'', that is, the capacity of proving theorems or having valid formulas.  Now the concept of sub-logic is sometimes associated with strength and sometimes associated with expressiveness, and sometimes with both  (e.g. in \cite{Beziau-CNCBEOH}), which is known to be the case of paradoxes \cite{Mossakowski-WILT}. Three kinds of systems are relevant here: model-theoretic logics, Tarskian and proof-theoretic logics, they will now be briefly defined. A logic $\L$ is called \emph{model-theoretic} if it is defined semantically and presented as a sequence ($\F,\M,\vDash$), where $\F$ is a set of formulas, $\M$ is a class of models and $\vDash$ is a satisfaction relation on $\M \times \F$. A logic $\L$ is \emph{Tarskian} if it is defined as ($\F,\vdash$), where $\vdash$ is a consequence relation on $\F$ (possibly multi-consequence). Finally, $\L$ is a proof-theoretic logic if it is defined as  $(\F, \mathcal{R})$, where $\mathcal{R}$ is a set of inference rules.\footnote{Some additional criteria are usually imposed for a  system to qualify as one of these three kinds, but they are immaterial here.}

In model-theoretic logics there is a straightforward approach to expressiveness that is also reasonably taken as a definition of logic inclusion: a logic $\L_2$ is at least as expressive/includes $\L_1$ if every class of structures characterizable in $\L_1$ is also characterizable in $\L_2$ (see e.g. \cite[p. 129]{Lindstrom-OCEL} and \cite{Barwise-Feferman-MTL}). This naturally only holds for logics defined within the same class of structures.  
If one wants also to compare logics defined within different classes of structures, then it does not seem adequate to use the concept of sub-logic, as we shall see below. It is better to use the concept of expressiveness.

There is no straightforward approach to expressiveness for Tarskian and proof-theoretic logics (TPL, for short). As for sub-logic, in TPL it is also linked with language and axiomatic extensions. However, we can often see ``sub-logic'' relations taken in a wider sense, i.e. when, for two given logics $\L$ and $\L'$, it happens that $\L'$ is not a language/axiomatic extension of $\L$, but there is a certain mapping of $\L$-formulas into $\L'$-formulas respecting the consequence relation. These cases are normally interpreted as saying that $\L$ is included\-/embeddable\-/reconstructible\-/interpretable\-/can be simulated in $\L'$. We propose to call these as expressiveness relations whenever they can be seen as modeling the following intuition

\vspace{1mm}
\begin{minipage}{0.1\textwidth}
	$(E)$ 
\end{minipage}
\begin{minipage}{0.8\textwidth}
	For every $\L$-sentence $\phi$, there is an $\L'$-sentence $\psi$ with the same meaning.
\end{minipage}
\vspace{1mm}

This same intuitive explanation of expressiveness holds for model-theoretic logics, and is used as a basis for formal criteria therein (e.g. \cite[p. 42]{Barwise-Feferman-MTL}). Thus we can have a reasonably homogeneous concept for comparing logics: that of expressiveness. We shall reserve the term ``sub-logic" just when there are axiomatic or language extensions, and we shall not use the term ``strength'' because it is ambiguous between  expressive and deductive strength.

A precise definition for the notion of relative expressiveness for model-theoretic logics was given already in the 1970s (e.g. in \cite{Lindstrom-OCEL} and \cite{Barwise-AAMT}). As we said, this definition is based on the capacity of characterizing structures and underlies each of the so-called Lindström-type theorems,\footnote{That is, theorems of the form ``If a logical system $\L'$ is at least as expressive as $\L$ and have properties $P_1,...,P_n$, then $\L'$ is as expressive as $\L$''; see e.g. \cite{Barwise-Feferman-MTL}, \cite{VanBenthem-TenCate-Vaananen-LTFFOL} and  \cite{Otto-Piro-CGFMLGM}.} which form the basis of abstract model theory.

\paragraph*{Single-class expressiveness}	Considering model-theoretic logics defined within the same class of structures, the above intuition can be captured easily since there is a common ground where sentences can be compared. This common ground is easily achieved by defining the meaning of a sentence $\phi$ in a logic $\L =(\F, \M, \vDash_{\L})$ as $\{\U \in \M\,|\, \U \vDash_\L \phi\}$ ($Mod_{\L}(\phi)$, for short). Thus we call this framework \emph{single-class expressiveness}.  Since every sentence in $\L_1$ is mapped to a sentence in $\L_2$ having the same meaning, this framework of expressiveness can be seen as consisting of certain formula-mappings between model-theoretic logics. A formal definition for it is then straightforward. Let $\tau$ be a signature and let $\L_1=(\F_1,\M,\vDash_{\L_1})$ and $\L_2=(\F_2,\M,\vDash_{\L_2})$ be model-theoretic logics.

\begin{definition}[$\expr_{EC}$]
$\L_2$ is at least as expressive as $\L_1$ ($\L_1 \expr_{EC} \L_2$) if and only if (iff, for short) for every $\tau-$sentence $\phi \in \F_1$ there is a $\tau-$sentence $\psi \in \F_2$ such that $Mod_{\L_{1}}(\phi)=Mod_{\L_{2}}(\psi)$.
\end{definition}

Notice that here the class of models $\M$ is the same for both $\L_1$ and $\L_2$, and $\phi,\psi$ share the same non-logical symbols. The above definition can be paraphrased in terms of elementary classes:\footnote{For some signature $\tau$, a class $\mathcal{K}$ of $\tau$-structures is elementary in a logic $\L$ iff there is an $\L$-sentence $\phi$ such that $\mathcal{K}=\{\U \,\,|\,\, \U \vDash_\L \phi\}$. A class $\mathcal{K}$ of $\tau$-structures is a projective class of $\L$ if for some $\tau' \supseteq \tau$ there is an $\L$-elementary $\tau'$-class $\mathcal{K'}$ such that $\mathcal{K}=\{\U' \restr \tau\,\, |\,\, \U' \in \mathcal{K'}\}$, where $\U' \restr \tau$ is the $\tau$-reduct of $\U'$.}  $\L \expr_{EC} \L'$ iff every elementary class of $\L$ is an elementary class of $\L'$.

Despite being the basis for many important results, $\expr_{EC}$ is very limited. It is not only restricted to model-theoretic logics, but it requires the classes of structures being compared to share the same signature. As a consequence, it only allows the comparison of logics defined within the same class of structures. The urge for a broader definition is not new.\footnote{See \cite[p. 358]{Tarlecki-BPTI}, \cite[p. 299]{Meseguer-GL}, \cite[p. 232]{Shapiro-FWF} and \cite[p. 130]{Chang-Keisler-MT}.} A straightforward means of extension already appears in \cite{Barwise-Feferman-MTL} and is examined in \cite{Shapiro-FWF}. Using the notion of projective class, one can loosen the above definition allowing that $\L'$ is at least as expressive as $\L$ iff every elementary class of $\L$ is a projective class in $\L'$ ($\L \expr_{PC} \L'$) (\ibid, p. 232).

Even among those expressiveness results using $\expr_{EC}$, we can notice some flexibility in its application.  One such example appears in \cite{Areces-et.al-EPML}, where the definition of $\expr_{EC}$ above is given, but afterwards (p. 307) it is informally relaxed in order to allow changes of signature, thus the proper definition being used appears to be the one based on projective classes ($\expr_{PC}$). The problem is that elsewhere we get different results depending on whether we use $\expr_{EC}$ or $\expr_{PC}$, as Shapiro showed \cite[p. 232]{Shapiro-FWF}: $\L(Q_0) \not \expr_{EC} \L(A)$ and $\L(A) \not \expr_{EC} \L(Q_0)$, but $\L(Q_0) \expr_{PC} \L(A)$ and $\L(A) \expr_{PC} \L(Q_0)$.\footnote{The logic $\L(Q_0)$ is the first-order logic extended with the quantifier ``there exists infinitely many", and $\L(A)$ is the first-order logic with the ``ancestral'' operation $A$, i.e. $Axy(Rxy)$ says that $x$ is an ancestor of $y$ in the relation $R$.}

Remaining within model-theoretic logics, a wider framework ---let us call it multi-class--- would comprise besides formula-mappings also structure-mappings, thus allowing structures of one logic to be mapped to structures of the other. This would enable the comparison of logics defined within different classes of structures. Recently there appeared two formal definitions of multi-class expressiveness, to wit \cite{Garcia-Matos-Vaananen-AMTFUL} and \cite{Kuijer-PHD}. In the sequence we will present them and argue that they are not adequate.

There have been also early claims outside abstract model-theory relating logics in the sense of (E) above, but no explicit definitions of the main concepts involved were given. Gödel used his result on the interpretation of classical into intuitionistic logic to infer that, contrary to the appearances, it is classical logic that is contained in intuitionistic logic \cite[p. 295]{Godel-1933e}. Since then, there followed many results of interpretations, embeddings, reconstructions,  simulations, etc. among Tarskian and proof-theoretic logics. Such results have often been used to justify some statement of inclusion or relative expressiveness between the logics at issue.\footnote{E.g. \cite[p. 154]{Thomason-RTLML-II}, \cite[p. 67]{Wojcicki-TLC}, \cite[p. 441]{Humberstone-CCL}, \cite[p. 163]{Humberstone-BTP}, \cite[p. 233]{Coniglio-TSNTBL}, \cite[p. 15]{Carnielli-Coniglio-Dottaviano-NDTBL} and the recent \cite[p. 207]{Agudelo-TNCLCL}.}  We proposed to call those with the underlying intuition ($E$) as expressiveness results. Naturally, this notion of expressiveness is no longer directly linked with the capacity of characterizing structures as in model-theoretic logics, rather it resides in the capacity of a logic to ``encode'' another. Let the framework of expressiveness based on such capacity  be named ``translational expressiveness''.\footnote{The term is borrowed from \cite{Peters-PHD}. Curiously, the same kind of problem appeared in computer science: there was a multitude of programming languages and process calculi and many informal claims relating the expressive power of such, through the existence of certain encodings of one into another. This situation fomented a series of works aiming at a standardization of such ``expressibility results'' (e.g. \cite{Felleisen-EPPL}, \cite{Parrow-EPA} and \cite{Gorla-TUAESRPC}). Though aimed at different objects, it is still possible to learn from this enterprise and propose the first steps of a standardization of a definition of relative expressiveness.}

As opposed to the case of model-theoretic logics, until recently there was no attempt to give a precise definition of relative expressiveness in this framework. To the best of our knowledge, Mossakowski et al. \cite{Mossakowski-WILT} were the first to give an explicit formal definition of translational expressiveness for logics, that is, an expressiveness relation  based on the existence of certain kinds of formula-mappings.  We will expose their definition and show that it is still not adequate. Then, some adequacy criteria for expressiveness are proposed and a formal criterion for translational expressiveness is given.

\subsection*{Structure of the paper}

This paper presents the following panorama on relative expressiveness between logics:

\begin{itemize}
	\item[(*)]  Relative expressiveness between logics (intuitive concept as given by (E))
		\begin{enumerate}
			\item[(a)] Adequacy criteria for expressiveness
			\begin{itemize}
				\item[$\rightarrow$] Approaches to (*) hopefully satisfying (a)
				\begin{itemize}[label={\rotatebox[origin=c]{180}{$\Lsh$}}]
					\item single-class
						\begin{itemize}[label={\rotatebox[origin=c]{180}{$\Lsh$}}]
							\item formal proposals: $\expr_{EC}$, $\expr_{PC}$
						\end{itemize}
					\item multi-class
						\begin{itemize}[label={\rotatebox[origin=c]{180}{$\Lsh$}}]
							\item formal proposals: $\expr_{gv}$, $expressiveness_g$
						\end{itemize}
					\item translational
						\begin{itemize}[label={\rotatebox[origin=c]{180}{$\Lsh$}}]
							\item formal proposals: Mossakowski et al.'s and $expressiveness_{gg}$.
						\end{itemize}
				\end{itemize}
			\end{itemize}
		\end{enumerate}
\end{itemize}

	In $\S 2$ the framework of multi-class expressiveness will be presented and two formal criteria will be analysed, one from \cite{Garcia-Matos-Vaananen-AMTFUL} ($\expr_{gv}$) and other from \cite{Kuijer-PHD} ($expressiveness_g$). We argue that, using the intuitive explanation of expressiveness given above, there are counterexamples to both. In the sequence, we investigate what is wrong with them and propose that moving to an even wider framework, encompassing a greater range of logics and lacking structure-mappings, might be promising.

	In $\S 3$ we present Mossakowski et al.'s formal criterion for translational expressiveness and show that, due to a result of \cite{Jerabek-UCT}, it is still not adequate. Then, some basic adequacy criteria for expressiveness will be proposed. In the sequence we analyse some formal conditions related to translations already appearing in the literature and investigate whether they satisfy the adequacy criteria. Finally, a formal sufficient condition for translational expressiveness ($expressiveness_{gg}$) is proposed. We will argue that $expressiveness_{gg}$ satisfies the criteria and is materially adequate.

\section{Multi-class expressiveness}

\subsection{M. García-Matos and J. Väänänen on sub-logic}

García-Matos and Väänänen gave a multi-class definition of sub-logic. Their definition is similar to one given in \cite{Meseguer-GL} but is laxer.\footnote{García-Matos and Väänänen's approach is a non-signature indexed version of the ``map of logics'' in \cite[p. 299]{Meseguer-GL}. In Meseguer's paper, it is not allowed for sub-logic mappings that sentences in the source logic be mapped to theories in the target logic, and the formula-mappings must be injective.}  Seemingly, they treat the term ``sub-logic'' as synonymous with ``expressiveness'' (exchanging the order of terms, naturally), since they present the Lindström theorems as being about sub-logic, whereas they are presented by one of the authors elsewhere as being about expressiveness (e.g. \cite{VanBenthem-TenCate-Vaananen-LTFFOL}). We shall argue that the relation defined must be seen as an expressiveness relation, and it will be shown that as an expressiveness relation, it has important downsides. Let us consider their definition of sub-logic \cite[p. 21]{Garcia-Matos-Vaananen-AMTFUL}:

\begin{definition}\label{GM-V-sublogic}
A logic $\L=(\F, \M, \vDash)$ is a \emph{sub-logic} of $\L'= (\F', \M', \vDash')$ (in symbols $\L \expr_{gv} \L'$) if there are a sentence $\theta \in \F'$ and functions $f: \M' \longrightarrow \M$, $\T: \F \longrightarrow \F'$ such that:
\begin{enumerate}
\item[(a)] For every $\U \in \M$ exists a $\U' \in \M'$ such that $f(\U')=\U$ and $\U' \vDash' \theta$
\item[(b)] For every $\phi \in \F$ and for every $\U' \in \M'$, if $\U' \vDash' \theta$, then ($\U' \vDash' \T(\phi)$ iff $f(\U') \vDash \phi$)
\end{enumerate}
\end{definition}

Thus, if the class of structures $\M'$ of a logic $\L'$ is richer than the class of structures $\M$ of a logic $\L$, one could still allow a comparison between $\L$ and $\L'$, by restricting $\M'$ to the translatable structures, i.e. those $\U'$ which  satisfy some condition $\theta$ and then use a function $f$ to translate this reduced class of $\L'$-structures into $\L$-structures.

\subsubsection{A problem with $\expr_{gv}$}

Let $\L=(\F, \mathcal{M}, \vDash)$ be a trivial propositional logic in some given signature, and let ($\mathfrak{M}, v$) be the set of is truth tables together with a valuation. Let $\L'=(\F', \mathcal{M'}, \vDash')$ be any logic that has at least one valid sentence $\delta$ and let the formula $\theta$ of the definition above be such $\delta$. Define the following mappings 

\begin{itemize}
	\item $f: \mathcal{M'} \longrightarrow \mathcal{M}$. For every $\U' \in \mathcal{M'}$, $f(\U')= (\mathfrak{M},v)$.
	\item $\T: \F \longrightarrow \F'$. For every $\phi \in \F$, $\T(\phi) = \delta$.
\end{itemize}
Then it is easily seen that both items (a) and (b) above are satisfied.

Thus, according to this definition of sub-logic, every logic containing at least one valid formula has a trivial sub-logic. If we think on the usual meaning given to ``sub-logic'', this not plausible at all, since the logic $(\F', \mathcal{M'}, \vDash')$ could be non-trivial and might even lack a trivializing particle, so how come it could have a trivial sub-logic?\footnote{This counter-example was based on another one given in \cite[p. 385-6]{Coniglio-Carnielli-TBLA}, which was given as an argument for strengthening the notion of translation used.} 

It is not enough to require that the mapping $\T$ be injective. Using an idea of \cite[p. 14]{Carnielli-Coniglio-Dottaviano-NDTBL}, take for target logic any $\L^*=(\F^*, \mathcal{M}^*, \vDash^*)$ that has a denumerable number of valid formulas $\delta_1, \delta_2, ...$ and define the mapping from the formulas of the trivial logic $\F=\{\phi_1, \phi_2, ...\}$ to $\L^*$-formulas as $\T(\phi_i)= \delta_i$. Still we have that $\L^*$ has a trivial sub-logic, once more, $\L^*$ may be any logic with a denumerable number of validities, also lacking a trivializing particle.

Naturally, the usual senses of logic inclusion, that is, through language or axiomatic extensions do not apply here. The only way to make sense of this is to interpret the above cases as saying that a trivial logic can be \emph{simulated} in any logic containing at least one validity. This capacity of simulating a logic is an expressive capacity, therefore the definition above is better seen as a definition of expressiveness. Yet, as an expressiveness relation, it is noteworthy that no restriction on the translation functions $f$ and $\T$ are imposed, so one may wonder whether the definition over-generates.

We are not in position to settle definitively this question. However we will give a plausibility argument to the effect that we should impose stricter conditions on model- and formula-mappings, since there is a natural and reasonable extension of the above definition that indeed over-generates. Though not, strictly speaking, a counter-example, the case to be presented below shall give evidence that there is an intrinsic problem with the above proposal for multi-class expressiveness.


As we said, the sentence $\theta$ on the above definition of $\expr_{gv}$ is intended to cut $\L'$-structures that are meaningless from the point of view of $\L$. Apparently, it would do no harm to the idea behind $\expr_{gv}$ to allow $\theta$ to be a recursive set of sentences, as it is normally done in works dealing with translations of logics and conversion of structures (e.g. \cite[p. 270]{Manzano-EFOL}). This would be useful if the logics at issue have no conjunction, so that $\theta$ could be a finite set of sentences; or if the low expressive power of the logics $\L$ and $\L'$ makes that the $\L'$-structures to be reduced into $\L$-structures be only characterizable through an infinite but recursive set of $\L'$-sentences. This happens in the case of many-sorted logic ($\mathcal{MSL}$) and $\FOL$. If $\theta$ is not allowed to be an infinite set of sentences, then $\mathcal{MSL}$ would not be a sub-logic, in the above sense, of $\FOL$, which is implausible. Though the conversion of $\FOL$-structures into $\MSL$-structures is mentioned \cite[p. 23]{Garcia-Matos-Vaananen-AMTFUL}, the case of a given $\FOL$-signature $\tau$ containing infinitely-many unary symbols $S_1,S_2,...$ is not considered. To convert  $\tau$-structures into $\mathcal{MSL}$-structures then one needs to make sure that unary predicates $S_1,S_2,..$ to be converted to many-sorted domains are non-empty. This would only be accomplished by setting $\theta=\{\exists x S_1(x), \exists x S_2(x), ...\}$ \cite[p. 260]{Manzano-EFOL}.

However, if one allows such modification another implausible situation occurs. Consider the classical propositional logic ($\CPL$) and a propositional logic $\WPL$, defined by Béziau \cite{Beziau-CNCBEOH}. $\WPL$ shares all the definitions of the classical propositional connectives, except for negation, where it has only one ``half'' of its clause: for a $\WPL$-model $M$ and formula $\phi$, if $M(\phi)= T$, then $M(¬\phi)=F$; the converse direction does not hold.

Béziau shows that there is a translation from $\CPL$ into $\WPL$. Below we will give Mossakowski et al.'s presentation of it, which includes also a model translation \cite[p. 107]{Mossakowski-WILT}. Given an n-ary connective $\#$, a translation $\T$ is literal for $\#$ if $\T(\#(\phi_1,...,\phi_n))= \#(\T(\phi_1),...,\T(\phi_n))$; for an atomic formula $p$, $\T$ is literal when $\T(p)=p$. Define the mapping ($\T,f): \CPL \longrightarrow \WPL$ as follows:

\begin{itemize}
	\item $\T: \F^{\CPL} \longrightarrow \F^{\WPL}$
	\begin{itemize}
		\item $\T(¬\phi)= \T(\phi) \rightarrow ¬(\T(\phi))$,
		\item literal for $\land,\lor,\rightarrow$ and atomic formulas;
	\end{itemize}

	\item and $f: \M^{\WPL} \longrightarrow \M^{\CPL}$ 

	\begin{itemize}
		\item $f(\mathfrak{M}^{\WPL},v) = (\mathfrak{M}^{\CPL},v)$,
	\end{itemize}
\end{itemize}
where $\mathfrak{M}$ comprises the truth-tables for each connective and $v$ a valuation on the propositional variables. Notice that $f$ takes a $\WPL$-model, keeps the valuation $v$ and replaces the truth-tables for the corresponding $\CPL$ ones.

Then we have that

\begin{theorem}[Mossakowski et al.]$ $\\

	$f(\mathfrak{M}^{\WPL},v) \vDash_{\CPL} \phi$ if and only if $(\mathfrak{M}^{\WPL},v) \vDash_{\WPL} \T(\phi)$.
\end{theorem}



	


The model mapping $f$ is surjective, so that it obeys (a) above.

Now Mossakowski et al. (\ibid, p. 100) define a mapping also from $\WPL$ to $\CPL$ using an auxiliary set of formulas $\Delta$ constructed out of $\CPL$-formulas.

Define the mapping $(\T',f', \Delta): \WPL \longrightarrow \CPL$ as follows:
\vspace{2mm}

\begin{itemize}
	\item $\T':\F^{\WPL} \longrightarrow \F^{\CPL}$
		\begin{itemize}
			\item For every $\phi \in \F^{\WPL}$, $\T'(\phi)=p_\phi$, where $p_\phi$ is a propositional variable.
		\end{itemize}
\end{itemize}

Define $\Delta$ as the following set of formulas, for $\phi, \psi \in \F^{\WPL}$:

\begin{multicols}{2}
	\begin{itemize}
		\item $\T'(\phi \land \psi) \leftrightarrow \T'(\phi) \land \T'(\psi)$
		\item $\T'(\phi \lor \psi) \leftrightarrow \T'(\phi) \lor \T'(\psi)$
		\item $\T'(\phi \rightarrow \psi) \leftrightarrow \T'(\phi) \rightarrow \T'(\psi)$
		\item $\T'(\phi) \rightarrow ¬\T'(¬\phi)$.
	\end{itemize}
\end{multicols}

	The purpose of $\Delta$ is to encode the semantics of $\WPL$ into the propositional variables $\{p_1,p_2,...\}$, since every $\WPL$-formula is translated into one of such $p_i$, in a $\CPL$-model satisfying $\Delta$ the valuation of the propositional variables $p_i$ is forced to respect the semantics of $\WPL$. For example, in $\WPL$, if $(\mathfrak{M}^{\WPL},v) \vDash_\WPL r$, then it holds that $(\mathfrak{M}^{\WPL},v) \not \vDash_\WPL ¬r$, but the converse direction does not hold.  This is simulated in the $\CPL$-models satisfying $\Delta$ by the fourth clause above: if $(\mathfrak{M}^{\CPL},v) \vDash_\CPL p_r$, then $(\mathfrak{M}^{\CPL},v) \vDash_\CPL ¬p_{¬r}$ which implies that $(\mathfrak{M}^{\CPL},v) \not \vDash_\CPL p_{¬r}$. But, as in $\WPL$, it does not hold that if $(\mathfrak{M}^{\CPL},v) \not \vDash_\CPL p_{¬r}$, then $(\mathfrak{M}^{\CPL},v) \vDash_\CPL p_{r}$.

Now define the model-translation $f': \M^{\CPL} \longrightarrow \M^{\WPL}$:
\begin{itemize}

	\item Let  $(\mathfrak{M},v)$ be a $\CPL$-model satisfying $\Delta$. Then $f'(\mathfrak{M},v)$ is defined as follows:
		\begin{itemize}
		\item For every $\WPL$-formula $\phi$, $f'(\mathfrak{M},v) \vDash_\WPL \phi$ iff $(\mathfrak{M},v) \vDash_\CPL \T'(\phi)$.
		\end{itemize}

\end{itemize}

$f'$ is also surjective (so it obeys (a) in the criterion for sub-logic above). Then we have that 

\begin{theorem}[Mossakowski et al.]$ $\\

	$f'(\mathfrak{M}^{\CPL},v) \vDash_{\WPL} \phi$ iff $(\mathfrak{M}^{\CPL},v) \vDash_{\CPL} \Delta$ and $(\mathfrak{M}^{\CPL},v) \vDash_{\CPL} \T'(\phi)$.
\end{theorem}	

Therefore, by the above results and according to the extended definition of sub-logic, we would have that $\WPL$ and $\CPL$ are one sub-logic of another, which is not plausible. $\CPL$ is not a sub-logic of $\WPL$ in the sense of language/axiomatic extension. Neither they are expressively equivalent, using $(E)$ above, since the ``half-negation'' present in $\WPL$ is not available in $\CPL$. 

The problem is that the translation from $\WPL$ to $\CPL$ uses a trick to sneak in the semantics of $\WPL$ into $\Delta$. Restricting the $\CPL$-models that satisfy $\Delta$, one simulates the behaviour of $\WPL$-formulas in the propositional variables $p_i$ and sustain such behaviour through the model-translation.

The modified version of $\expr_{gv}$, allowing $\theta$ to be a recursive set of sentences looks at least as ``natural'' as the original one. Even considering the original definition  \ref{GM-V-sublogic} we can see that there is something wrong with it, in not requiring any kind of preservation of the structure of formulas e.g. by forcing $\T$ to be inductively defined through the formation of formulas. Then one may conjecture that, among more expressive logics, there be translations ($\T,f$) where $\T$ maps entire formulas $\phi$ to propositional variables $p_\phi$ and, with \emph{a sentence} $\theta$ restricting the target structures, $f$ is able to mimic the semantic behavior of $\phi$. Then it is very doubtful that the obtained $p_\phi$ would have the same meaning as $\phi$.

Thus, we think we have good reasons to consider that García-Matos and Väänänen's definition of sub-logic is not adequate. It would certainly be better to use a stronger notion of translation, paying attention to the structure of formulas. Only then the meaning of the target-formulas could be said to match the meaning of the source-formulas. Below we will see that a development along this line appeared in the literature.  Nevertheless, there is still a structure-attentive translation that ``cheats'' similarly as the one above,  mimicking the semantics of one logic into the other.

\subsection{L. Kuijer on multi-class expressiveness}

In his doctorate thesis \cite{Kuijer-PHD} Kuijer studies the expressiveness of various logics of knowledge and action, these logics are taken in the model-theoretic sense.  He notices that there are some results relating  logics similarly as in single-class expressiveness.\footnote{The referred results are: \cite{Thomason-RTLML-I}, \cite{Gasquet-Herzig-FCNML}, \cite{Goranko-Jamroga-CSLMAS}, \cite{Broersen-Herzig-Troquard-2006a} and \cite{Broersen-Herzig-Troquard-2006b}.} These works were selected as prototypical for a criterion in the wider framework of multi-class expressiveness.

The purpose is to investigate features shared by all the results and construct a criterion, to be called ``$expressiveness_g$'', based on these features.  Similarly with the work of García-Matos and Väänänen exposed above, these prototypes involve translations of sentences and translations of structures. So a translation from $\L_1$ to $\L_2$ is a pair ($\T,f$), with $\T: \F_1 \rightarrow \F_2$ and $f: \M_1 \rightarrow \M_2$ or $f: \M_2 \rightarrow \M_1$, such that $(\T,f)$ satisfies some given conditions. 

A first plausible condition is that ($\T,f$) must preserve and respect truth:
\begin{definition}[Truth preserving]
A translation $(\T,f): \L_1 \rightarrow \L_2$  with  $\T: \F_1 \rightarrow \F_2$ and $f: \M_1 \rightarrow \M_2$ is truth preserving if, for every $\phi \in \F_1$ and $\U \in \M_1$
\begin{center} $\U \vDash_{\L_1} \phi$ if and only if  $f(\U) \vDash_{\L_2} \T(\phi)$. \end{center}
\end{definition}
Then a tentative definition of $expressiveness_g$ could be 
\begin{quote}
$\L_2$ is at least $expressive_g$ as $\L_1$ iff there is a $(\T,f): \L_1 \rightarrow \L_2$ that is truth preserving.
\end{quote}
The problem is that the requirement of truth preservation is very weak, indeed there are several trivial truth-preserving translations among almost every logic. Kuijer gives the following example \cite[p. 88]{Kuijer-PHD}.

\subsubsection{A trivial translation}

Let $\L_1 = (\F_1, \M_1, \vDash_{\L_1})$ be any logic on possible world semantics such that $\F_1$ is countable and let $\L_2=(\F_2, \M_2, \vDash_{\L_2})$ be a logic where $\F_2$ is a countable set of propositional variables but with no connectives and where $\M_2$ is a class of models with possible worlds. Thus, every $\U' \in \M_2$ is a set of possible worlds with a valuation.

Define a truth-preserving translation ($\T_t, f_t$) from $\L_1$ to $\L_2$ in the following way: map every $\phi \in \F_1$ to a propositional variable $p_\phi \in \F_2$, $f_t$ maps a model $\U \in \M_1$ to a model $\U' \in \M_2$ taking the set of possible worlds of $\U$ and removing every other structure, and with the following valuation $v(p_\phi)= \{ w \in \U \,\,|\,\, (\U, w) \vDash_{\L_1} \phi\}$. Then clearly, by definition, $(\T_t,f_t): \L_1 \longrightarrow \L_2$ is a truth preserving translation.  

\subsubsection{Defining $expressiveness_g$}

Since $\L_1$ in the above example is an arbitrary logic on possible world models, if truth preservation were the only condition for multi-class expressiveness, $\L_2$ would be at least as expressive as $\L_1$, which is absurd, given that $\L_2$ has scarce expressive means. Nevertheless, truth-preservation is clearly a necessary condition. Thus, one must find other features $P_1,...,P_n$ a translation must satisfy in order to serve as a formal elucidation of the notion of multi-class expressiveness. 

Another immediate criterion that comes to mind in order to avoid the trivial translations is to require the preservation of validities and entailment relations. However, some of the chosen prototypical translations do not preserve validity and some do not preserve entailment. Since the idea was to capture the essential features shared by all prototypical translations in  $expressiveness_g$, none of these can be imposed as a necessary condition.

Kuijer then goes through a number of tentative criteria, e.g. preservation of atomic formulas, of sub-formulas, etc., and shows that they are either too lax or too restrictive. Among the lax criteria, that is, the ones that are satisfied by some trivial translation, is one that Kuijer considers nonetheless important, the criterion of being model based:
\begin{definition}[Model based]
A translation $(\T,f)$ is model based if there are two functions $f_1,f_2$ such that, for all $(\mathfrak{M},w) \in \M_1$, we have that $f(\mathfrak{M},w)= (f_1(\mathfrak{M}),f_2(\mathfrak{M},w))$.
\end{definition}
A model based translation would force $f$ to preserve some structure of $\mathfrak{M}$ and  prevent that the pointed models $(\mathfrak{M},w)$ and $(\mathfrak{M},w')$ be translated to completely unrelated models.
 
Finally, the condition that apparently divides the good from bad translations and gives a reasonable notion of multi-class expressiveness is the criterion of being finitely generated.  For the sake of simplicity, some aspects of the definition below are not completely formalized.\footnote{For the complete formal definition, the reader may consult \cite[p. 115]{Kuijer-PHD}.}  Let $\F$ be a set of formulas generated by a set $\mathcal{P}$ of propositional variables and a set $\mathcal{C}$ of connectives. Let $X=\{x_1,x_2,...\}$ be a set of variables with $\mathcal{P} \cap X = \emptyset$, and let $\F^{X}$ be the set of formulas generated by $\mathcal{P} \cup \{x_1,x_2,...\}$ with the connectives $\mathcal{C}$. Then we have (\ibid, p. 115):

\begin{definition}[Finitely Generated]
Let $\L_1$ and $\L_2$ be such that $\F_i$ is generated by a set $\mathcal{P}_i$ of propositional variables and a finite set  $\mathcal{C}_i$ of connectives, for $i \in \{1,2\}$. Let $\phi^X \in \F^{X}_1$, then a translation $(\T,f): \L_1 \rightarrow \L_2$ is \emph{finitely generated} if $\T$ can be inductively defined by a finite number of clauses of the form
\begin{center}     $\T(\phi^X)= \psi^X$ for $(x_1,...,x_n) \in \Psi$ \end{center}
where $\psi^X$ is an $\F^{X}_2$-sentence constructed out of $x_1,...,x_n$ and possibly containing $\T(x_i)$, for $x_i \in \F^{X}_1$; and where $\Psi$ is the range of  the $x_i$, e.g. if a given $x_i$ is to be replaced by a formula or only by an atomic formula.
\end{definition}

The set $X$ contains the special propositional variables to be used in the translation clauses, for which one can substitute formulas. An example of such a translation clause is: $\T(x_1 \rightarrow x_2) = ¬(\T(x_1) \land ¬\T(x_2))$ for $(x_1,x_2) \in\,\, \F_1 \times \F_1$; and $\T(x_1)=x_1$ for $x_1 \in \mathcal{P}$.

The idea is that (\ibid, p. 110) it is the fact of being inductively defined and thus respecting (some) of the structure of the formulas that sets the finitely generated translations apart from the trivial translations. Thus Kuijer concludes that the truth-preserving  translations giving rise to an expressiveness relation could be characterized as the ones being finitely generated and model-based. Therefore, the final criterion given for multi-class expressiveness is (\ibid, p. 111)

\begin{definition}[Expressiveness$_g$]
Let $\L_1$ and $\L_2$ be such that $\F_i$  is generated by a set $\mathcal{P}_i$ of propositional variables and a finite set  $\mathcal{C}_i$ of connectives for $i \in \{1,2\}$.

Then $\L_2$ is at least as $expressive_g$ as  $\L_1$ iff there is a translation $(\T,f)$ from $\L_1$ to $\L_2$ that is model based, finitely generated and truth preserving.
\end{definition} 

\subsubsection{A problem with $expressiveness_g$}
	\label{problem-expressiveness-g}

Kuijer had no pretensions that his multi-class definition were to be \emph{the} generalization of expressiveness as given by the single-class framework. The aim was to find only \emph{a} ``reasonable generalization'' (\ibid, p. 83). While keeping this in mind, we would like to argue that his proposal is still not good enough as a criterion for multi-class expressiveness. This is because one can find a pair of logics $\L$, $\L'$ such that $\L'$ is intuitively \emph{more} expressive than $\L$, although  $\L$ is at least as $expressive_g$ as $\L'$.

The logics at issue are Epstein's relatedness logic ($\R$) \cite[p. 80]{Epstein-SFL} and classical propositional logic ($\CPL$). The logic $\R$ besides the truth-functional connectives, has a relevant implication ``$\rightarrow$", which is the reason it is intuitively more expressive than $\CPL$, which lack such a connective.  The referred translation would imply that $\CPL$ is at least as $expressive_g$ as $\R$.

Despite the circumscribed character of Kuijer's criterion, we think that a reasonable generalization of single-class expressiveness should be able to deal with a reasonable amount of logics, not only with a handful of them. Particularly when the logics at issue are in the literature, and have not been constructed in an ad-hoc fashion just to give a counter-example. Finally, there is nothing specific about the logics appearing in the counter-example, so it is quite possible that there are also modal counter-examples.

Epstein presents $\R$ with the connectives $¬,\land, \rightarrow$. The first two are defined as usual and the underlying idea for interpreting the relevant implication symbol ``$\rightarrow$" is as follows. It holds that $p \rightarrow q$ whenever $p$ materially implies $q$ and both are subject-matter related to each other through a relation $\mathscr{R}$  defined on all propositional variables. Specifically, for propositional variables $p_i,p_j$ and $\R$-sentences $\phi$ and $\psi$, $\mathscr{R}(\phi,\psi)$ holds if and only if for some $p_i$ occurring in $\phi$, and $p_j$ occurring in $\psi$, it holds that $\mathscr{R}(p_i,p_j)$. Thus, the truth table for ``$\rightarrow$" is the one for material implication with an additional column for $\mathscr{R}$, so that if $\mathscr{R}(\phi,\psi)$ holds and $¬(\phi \land ¬\psi)$ is true, then $\phi \rightarrow \psi$ is true; else, if $\mathscr{R}(\phi,\psi)$ does not hold, then $\phi \rightarrow \psi$ is false.

 Let $\tau= \{p_0,p_1,..., ¬, \rightarrow, \land\}$ be a signature for $\R$. An $\R$-model ($\mathfrak{M},\mathscr{R},v$) is formed by the truth-tables for $\land, ¬, \rightarrow$, a symmetric and reflexive relation $\mathscr{R}$ on $\tau$-formulas and a valuation $v$.  For propositional variables $d_{i,j}$, let $\tau^+ = \{p_1,p_2,...\} \cup \{d_{i,j} \,\,|\,\, i,j \in \mathbb{N}\} \cup \{¬,\land, \supset\}$. Let $\CPL$ be defined on $\tau^+$ (note we use $\supset$ here to emphasize that it is a material implication).\footnote{The use of new propositional variables is for the sake of simplicity, as we could arrange the $p_1,p_2,...$ in $\CPL$ so as to assign some of the $p_i$s the role of such $d_{i,j}$.}

 We will see below that there is a truth-preserving, model-based and finitely generated translation $(\T^E, f^E): \R \longrightarrow \CPL$. The mapping $\T^E$ is defined as follows:\footnote{The mapping presented was adapted from (\ibid, p. 299). It was given a simpler form which makes the proof of the theorem below straightforward.  We refer to Epstein's mapping as $\T^{E^*}$, which is identical with $\T^{E}$ except for $\rightarrow$, where 

\vspace{1mm}
 $\T^{E^*}(\phi \rightarrow \psi)=$\\

 $(\T^{E^*}(\phi) \supset \T^{E^*}(\psi)) \land [ (\underset{p_i \text{ in } \phi,\, p_j \text{ in } \psi}{\bigvee} d_{i,j}) \,\,\lor\,  (\underset{p_n \text{ in } \phi,\, p_n \text{ in } \psi}{\bigvee} (d_{n,n} \lor ¬ d_{n,n}))]$. 
 \vspace{1mm}

 Notice that our mapping $\T^E$ below is only truth-preserving while Epstein's $\T^{E^*}$ is also validity-preserving, as e.g. $\T^E(p \rightarrow p) = (p \supset p) \land d_{p,p}$ and $\T^{E^*}(p \rightarrow p)= (p \supset p) \land [d_{p,p} \lor (d_{p,p} \lor ¬d_{p,p})]$.}

\begin{itemize}
	\item $\T^E(\phi \rightarrow \psi) = (\T^E(\phi) \supset \T^E(\psi)) \land d_{\phi,\psi}$
	\item literal for $¬,\land$ and atomic formulas.\footnote{Kuijer requires also  that no propositional variable occurs outside the scope of a translation function, so for atomic formulas  one should use  additional functions $s: \mathcal{P} \longrightarrow \mathcal{P}$. Thus we can take the identity function as such $s$.}

\end{itemize}

Here the basic idea for the translation of $\phi \rightarrow \psi$ comes from the definition of ``$\rightarrow$'': $\phi$ materially implies $\psi$ and both formulas are related through $\mathscr{R}$. As the translation is defined inductively through the formation of formulas by a finite number of clauses, it is \emph{finitely generated}.

Now, from an $\R$-model ($\mathfrak{M}, \mathscr{R}, v$), one easily defines a transformation $f^E$ from $\R$-models to $\CPL$-models. Let $f^E(\mathfrak{M}, \mathscr{R}, v)= (\mathfrak{M}^*, v^*)$, where, for $\mathfrak{M}^*$ take all the truth-tables  in $\mathfrak{M}$, excluding the one for $\rightarrow$. Define $v^{*}$ as follows (adapted from \cite[p. 300]{Epstein-SFL}):

\begin{itemize}
	\item $v^{*}(p_i)= v(p_i)$;
	\item $v^*(d_{\phi,\psi}) = T$ iff $\mathscr{R}(\phi,\psi)$ holds.
\end{itemize}
Clearly $f^E$ is \emph{model-based}. 

Both $\CPL$ and $\R$ satisfy a semantic deduction theorem (\ibid, p. 299). To prove that $(\T^E,f^E)$ is {truth-preserving}, one has to prove only that, for an arbitrary $\R$-model $(\mathfrak{M}, \mathscr{R},v)$, it holds that

\begin{theorem}[adapted from Epstein]
$ $ \vspace{2mm}

$(\mathfrak{M}, \mathscr{R},v) \vDash_\R \phi$ if and only if $f^E(\mathfrak{M}, \mathscr{R},v) \vDash_\CPL \T^E(\phi)$.
\end{theorem}

\begin{corollary}
	$\CPL$ is at least as $expressive_g$ as $\R$.
\end{corollary}

The main question now is: does $(\T^E,f^E): \R \longrightarrow \CPL$ show that $\CPL$ is at least as expressive as $\R$? We do not think it is reasonable to say so, since the extra expressiveness brought about by the implication connective in $\R$ is only by a trick mimicked in $\CPL$. Independently of the model-translation $f^E$ to give the intended truth values for the ``relevance-mimicking" variables $d_{\phi,\psi}$, it is not possible to have a relevant conditional in $\CPL$, by say, adjoining to a conditional $\phi \supset \psi$ such variables $d_{\phi,\psi}$. To do so, would require too much for the intended meaning of such variables. Surely this would not augment the expressive power of the propositional logic, as it concerns only an interpretation of propositional variables, and intuitively, specific interpretations of propositional variables do not influence the expressiveness of a logic.

Anyway, the model-mappings are not essential for these translations using indexed variables,  they only facilitate their definition. An early example was given by Richard Statman in \cite{Statman-IPLPSC} where a translation of $\IPL$ into its implicational fragment $\IPL^{\restr \{\rightarrow\}}$ is presented. There, the conjunctions $p \land q$ are mapped to implications containing $x_{p\land q}$, among formulas of the sort $x_p \rightarrow (x_q \rightarrow x_{p \land q})$, $x_{p \land q} \rightarrow x_p$, etc. Here the situation is entirely different since the proof-theoretic behaviour of individual conjunctions are encoded in specific variables using implicational axioms.

Coming back to Kuijer's criterion, we argued above that it is not enough to give an intuitively adequate account of expressiveness. If the model mapping were not from the source logic to the target logic but vice-versa, then there would not be such truth preserving mappings from $\R$ to $\CPL$, as there would be no way to construct the relatedness predicate $\mathscr{R}$ out of a $\CPL$-model. Kuijer discarded such a definition of the model mappings $f$ since it  implies that any truth-preserving translation is also validity preserving,\footnote{Suppose that for logics $\L=(\F, \M, \vDash_\L)$ and $\L'=(\F',\M'\vDash_{\L'})$ that $(\T,f): \L \longrightarrow \L'$ is truth-preserving, with $\T:\F \longrightarrow \F'$ and $f: \M' \longrightarrow \M$. Suppose $\phi$ is $\L$-valid, then for any model $\U' \in \M'$, $f(\U') \vDash_\L \phi$, thus, by truth-preservation, $\U' \vDash_{\L'} \T(\phi)$, but $\U'$ is any $\L'$-model, thus, $\T(\phi)$ is $\L'$-valid.}  and some of his paradigmatic examples of multi-class expressiveness are not validity preserving.

Let us analyse a possible  strengthening on the formula translation. We will not give a detailed analysis of features of translations since it suffices to notice that Epstein's translation preserves completely the structure of the formulas, except for $\rightarrow$. For this case, additional propositional variables $d_{\phi,\psi}$ must be introduced to bear the intended meaning of $\mathscr{R}$ (variables whose interpretation in $\CPL$ is sustained by the model translation.) If one required that $\T$ be \emph{compositional}, that is,  every $n$-ary connective $C(\phi_1, ..., \phi_n)$ of the source logic is translated by a schema $C^\T(\T(\phi_1)/\xi_1,...,\T(\phi_n)/\xi_n)$ of the target logic, then the above translation would not pass the test. This is because $p_1 \rightarrow p_2$ is translated through the schema $¬(\xi_1 \land ¬\xi_2) \land d_{p_1,p_2}$, and $p_3 \rightarrow p_4$ by the schema $¬(\xi_1 \land ¬\xi_2) \land d_{p_3,p_4}$. If the translation were compositional, dealing with the same connective, the same translation schema would be used.

The problem of adopting this criterion is that it implies that the connectives be translated one at a time, and again some of the paradigmatic translations selected by Kuijer takes into consideration sequences of connectives, so they would not satisfy it.

Therefore, to prevent translations such as those above from passing the test for multi-class, one would have to use a criterion for $\T$ that is stronger than being finitely generated, but weaker than being compositional. Nevertheless, the enterprise of placing restrictions on the formula translations $\T$ alone seems not to be promising, as the model-translations play a major role in the counter-examples presented above. On the other hand, placing also restrictions on model-translations and making them fit with the restrictions on formula-translations is a very complex enterprise, and there may be better alternatives.

Given this situation, we would like to suggest a change of perspective as regards relative expressiveness between logics. Below, some comments will be made regarding the nature of the notion of expressiveness and its relation with the concept of logical system it applies to.

\subsection{Single-class expressiveness vs multi-class expressiveness vs translational expressiveness}

Now we would like to make some remarks on the study of the relation of expressiveness between logics. As we commented before, in the single-class framework it is very simple to define relative expressiveness, since there is a common ground, the structures, where one can compare whether the sentences have the same meaning. Now consider the multi-class framework, if $\L=(\F,\M,\vDash)$ and $\L'=(\F',\M',\vDash')$ are defined on different classes of structures, how would we know whether an $\L$-sentence $\phi$ and an $\L'$-sentence $\psi$ have the same meaning? After all, in this case it trivially holds that $Mod_\L{(\phi)} \not = Mod_{\L'}(\psi)$. 

As we saw, for this task new tools are needed: a model-mapping $f: \mathcal{M} \longrightarrow \mathcal{M'}$ or $f': \mathcal{M'} \longrightarrow \mathcal{M}$;\footnote{The translation $f$ presupposes a mapping $\sigma$ of signatures: for each $\L[\tau]$-structure, there would correspond a $\L'[\sigma(\tau)]$-structure, respectively for $f'$.} and a formula-mapping $\T: \F \longrightarrow \F'$ or $\T': \F' \longrightarrow \F$. Now, for an $\L$-formula $\phi$ and $\L'$-formula $\psi$, we would have some possibilities for guessing when $\phi$ and $\psi$ have the same meaning:

\begin{multicols}{2}
\begin{itemize}
	\item $Mod_\L(\phi) = Mod_\L(\T'(\psi))$,
	\item $Mod_{\L'}(\psi)= Mod_{\L'}(\T(\phi))$, 
	\item $f[Mod_\L(\phi)]= Mod_{\L'}(\psi)$,\footnote{Let $f[Mod_\L(\phi)]= \{f(\U) \,\,|\,\, \U \in Mod_\L(\phi)\}$.}
	\item $f'[Mod_{\L'}(\psi)] = Mod_\L(\phi)$.
\end{itemize}
\end{multicols}

Thus, now the weight goes on the notion of translation $(\T,f)$. As we saw in the examples presented above, for ($\U,\phi$) in $\L$, and ($\U',\psi$) in $\L'$, the task of establishing the congruence between the pairs ($\U$, $\phi$) and ($\U'$, $\psi$)  by means of translations is very difficult. Basing it on satisfaction is far away from being sufficient, since we can easily devise translation functions such that $\U$ satisfies $\phi$ iff $\U'$ satisfies $\psi$.

On the other hand, imposing conditions on $(\T,f)$ is a complex enterprise, because either it under-generates or, by a little breach, it over-generates. Moreover, the need to have model-mappings besides formula\--map\-pings may open up a back door to undesirable translations, to see it, consider again the examples offered against García-Matos \& Väänänen's and Kuijer's approaches. All of them use some ``trick'' in the formula-translation function and sustain it through the model-translation. Then it is of little help to place structural restrictions on formula-translations, as did Kuijer. He also tried placing restrictions on model-translations, but it did not help either.

	Therefore, it might be more promising to move to a wider framework of relative expressiveness, dispensing with the semantic notions altogether.  In this framework, to be called ``translational expressiveness'', we would then concentrate the investigations on the conditions on formula translations. The aim is to find the set of conditions that better preserve/respect the theoremhood/consequence relation and the structure of formulas of each logic. This way a reasonable formal criterion of expressiveness for Tarskian and proof-theoretic logics (TPL, for short) would be obtained, and a bigger range of logics would be comparable. Finally, these advantages would arguably come at no cost, since this wider enterprise would not be more difficult than multi-class expressiveness.

	The big difference between the approaches of expressiveness is not in the division between expressiveness for model-theoretic logics and for TPL, but in the division, \emph{in} model-theoretic logics, of expressiveness within the same and within different classes of structures. Naturally the most direct concepts of expressiveness are linked with the capacity of characterizing structures, but this only applies when comparing the \emph{same} class of structures.
	
	If one allows translations between structures, such capacity is no longer at issue.  Once we depart from the safe harbour of a single class of structures for comparing logics, then all bets are off. Multi-class expressiveness does not guarantee a firmer grasp of the intuitive concept of expressiveness anymore than translational expressiveness. Since the move to a wider framework might not only free us from problems inherent to multi-class expressiveness, but also allow a bigger range of comparison of logics, then the prospects for the enterprise are better.

	As we said in the introduction, people have been using informally some concepts of translational expressiveness between logics. However, as opposed to what happens with model-theoretic logics, to the best of our knowledge, in the literature there is only one explicit and formal criterion in this framework, that of \cite{Mossakowski-WILT}. In the next section, their proposal will be analysed and we will show that it is not adequate. We shall then propose some adequacy criteria for expressiveness and a formal criterion in the framework of translational expressiveness will be given. We then argue that the criterion satisfies the adequacy criteria.

\section{Translational expressiveness: obtaining a still wider notion of expressiveness}

	In this section we will deal with logics in the Tarskian and proof-theoretic sense. We also mention logics taken as a closed set of theorems/validities, to be called simply ``formula logics". Let $\L_1$ and $\L_2$ be  logics, $\Gamma \cup \{ \phi\}$ be a set of $\L_1$-formulas and $\T$ a translation mapping $\L_1$-formulas into $\L_2$-formulas in such a way that for each $\L_1$-formula $\phi$:

	\begin{center}
	$\vdash_{\L_1} \phi$ if and only if $\vdash_{\L_2} \T(\phi)$.
	\end{center}
	In this case $\L_1$ is translatable into $\L_2$ with respect to theoremhood.

	If it is the case that
	\begin{center}
	$\Gamma \vdash_{\L_1} \phi$ if and only if $\T(\Gamma) \vdash_{\L_2} \T(\phi)$
	\end{center}
	then  $\L_1$ is translatable into $\L_2$ with respect to derivability  \cite[p. 216]{Prawitz-Malmnas}. The later translations are known as \emph{conservative} translations \cite{Feitosa-Dottaviano-CT}.

\begin{definition}[Conservative translation]
	A conservative translation is a translation with respect to derivability.
\end{definition}

Whenever we want to refer indistinctly to translations with respect to theoremhood or conservative translations, the term \emph{back-and-forth} will be employed.

\begin{definition}[Back-and-forth translation]
A translation is back-and-forth if it is either a theoremhood preserving or a conservative translation.
\end{definition}

\subsection{Mossakowski et. al.'s approach}

As far as we know, Mossakowski et al. \cite{Mossakowski-WILT} proposed the first explicit criterion for the concept of sub-logic and expressiveness in the framework of translational expressiveness:

\begin{definition}[Sub-logic]
$\L_1$ is a sub-logic of $\L_2$ if and only if there is an  injective conservative translation from $\L_1$ to $\L_2$;
\end{definition}

\begin{definition}[Expressiveness]
$\L_1$ is at most as expressive as $\L_2$ iff there is a conservative translation $\alpha: \L_1 \longrightarrow \L_2$.
\end{definition}

The authors do not explain why sub-logic requires injective conservative mappings while expressiveness does not. Anyway, we will see that these criteria for sub-logic and expressiveness via conservative mappings do not work. 

The conception that conservative translations could give rise to a notion of expressiveness and also a notion of logic inclusion has been supported more than once. For example, in  \cite[p. 233]{Coniglio-TSNTBL} it is said that the existence of a conservative translation (maybe injective or bijective) would give rise to some kind of logic inclusion between Tarskian logics.\footnote{The author says (\ibid):
\begin{quotation}
If we assume (...) a Tarskian perspective, then a logic system is nothing more than a set of formulas together with a [consequence] relation (...) Thus, the preservation of that relation by a conservative translation [from $\L_1$ to $\L_2$] would reveal that, as structures,  $\L_2$ ``contains'' $\L_1$ (Probably we should add the requirement that $f$ is an injective or even a bijective mapping.)
\end{quotation}
} Also for Kuijer, conservative translations give an adequate concept of expressiveness for Tarskian logics \cite[p. 86]{Kuijer-PHD}.\footnote{The author says (\ibid):
\begin{quotation}
There is a conservative translation from $\L_1$ to $\L_2$ if and only if everything that can be said in $\L_1$ can also be said in $\L_2$. 
\end{quotation}
}

Unfortunately, conservative translations will not make a reasonable concept neither of sub-logic nor of expressiveness. Due to a result of Je\v{r}ábek \cite{Jerabek-UCT}, explaining expressiveness and sub-logic through conservative translations would make $\CPL$ include and be at least as expressive as many familiar logical systems, e.g. first-order logic. He proved the following result (\ibid, p. 668), where for a logic $\L$, a translation is most general whenever it is equivalent to a substitution instance of every other translation of $\L$ to $\CPL$.

\begin{theorem}[Je{\v r}ábek]
For every finitary deductive system $\L= (\F, \vdash)$ over a countable set of formulas $\F$, there exists a conservative most general translation $\T: \L \rightarrow \CPL$. If $\vdash$ is decidable, then $f$ is computable. 
\end{theorem}

The defined mapping is injective.\footnote{For the sake of brevity, we omit the definition of the translation and simply point out that it is a non-general-recursive translation (to be defined below).} Let a logic be called ``reasonable'' if it is a countable finitary Tarskian logic. Je\v{r}ábek managed to generalize even more his results so that almost any reasonable logic can be conservatively translated into the usual logics dealt with in the literature.\footnote{Among others, classical, intuitionistic, minimal and intermediate logics, modal logics (classical or intuitionistic), substructural logics, first-order (or higher-order) extensions of the former logics.}  Now one would hardly accept that every  countable finitary logic has the same expressiveness or is one sub-logic of the other.

The author criticizes the notion of conservative translation for not requiring the preservation of neither the structure of the formulas nor the properties of the source logic \cite[p. 666]{Jerabek-UCT}. Thus, it must be strengthened in order to serve for an expressiveness measure.  This could be done in a simpler way by requiring injective, surjective or bijective mappings. As Je{\v r}ábek's mapping is injective, only requiring injectiveness will not do. As a matter of fact, it seems that already requiring injectiveness one is overshooting the mark. Since in this way $\CPL^{\restr \{\land, ¬\}}$ would not be as expressive as $\CPL^{\restr\{\land, ¬, \lor\}}$. Any mapping $g:\CPL^{\restr\{\land, ¬, \lor\}} \longrightarrow \CPL^{\restr\{\land, ¬\}}$ would have to map both $\CPL^{\restr\{\land, ¬, \lor\}}$-sentences $\phi \lor \psi$ and $¬(¬\phi \land ¬\psi)$ to the same $\CPL^{\restr\{\land, ¬\}}$-sentence $¬(¬g(\phi) \land ¬g(\psi))$, so it would not be injective.


Other kinds of strengthening hinted by  Je{\v r}ábek's (\ibid) are: 
\begin{enumerate}
\item force the mappings to preserve more structure of the source logic sentences in the target logic;
\item force the mappings to preserve more properties of the source logic.
\end{enumerate}

The adequacy criteria for expressiveness to be given below will require to some extent (1) and (2).

\subsection{Adequacy criteria for expressiveness}

As we saw above, Mossakowski et al. \cite{Mossakowski-WILT} gave a proposal for a wide notion of expressiveness: by means of the existence of conservative translations. Due to Je{\v r}ábek's results on the ubiquity on this kind of translation, their definition is not adequate. Maybe we should step back and think about some adequacy criteria every approach to expressiveness ought to accomplish.

	The intuitive explanation for expressiveness $(E)$ given in the beginning elucidates relative expressiveness in terms of a certain congruence of meanings. It appears already in a more direct form in Wójcicki's Theory of Logical Calculi \cite[p. 67]{Wojcicki-TLC}, and we place it as the first adequacy criterion

	\begin{quote}[\textbf{Adequacy Criterion 1}] $\L_2$ is at least as expressive as $\L_1$ only if everything that can be said in terms of the connectives of $\L_1$ can also be said in terms of the connectives of $\L_2$.\end{quote}

Here, for ``being said in terms of the connectives'' there can be stricter interpretations (as proposed by Wójcicki, Humberstone, Epstein) and wider interpretations (as proposed by Mossa\-kow\-ski et al. and us), to be developed below.

There are some meta-properties of logics that are intuitively known to limit or increase expressiveness. Thus, the presence/absence of such properties can be used to test whether there can be or not an expressiveness relation between the given logics. A first one coming to mind is that nothing can be expressed in a trivial logic, so it cannot be more expressive than any logic. Another one has to do with the relation between expressiveness and computational complexity.  This relation has even been stated as the ``Golden Rule of Logic'' by van Benthem in \cite[p. 119]{VanBenthem-WILGSI}, where he says ``gains in  expressive power are  lost  in  higher complexity''. Nevertheless, the ``Golden Rule'' is not quite useful here, since we know that in general neither a low expressiveness means low complexity,\footnote{For example, there are propositional logics whose complexity  is in  each arbitrary degree of unsolvability (e.g. see \cite{Gladstone-SWCPC}).} nor a high complexity means high expressiveness.\footnote{There can be equally expressive logics that, though both decidable, have very different computational complexities (e.g. see \cite{Levesque-Brachman-ETKRR}).}

Nevertheless the complexity levels of decidability/undecidability can be useful for expressiveness comparisons: if a logic is decidable, then it cannot describe Turing machines, Post's normal systems, or semi-Thue systems. Therefore, a decidable logic $\L$ cannot be \emph{more} expressive than an undecidable logic $\L'$, otherwise, $\L$ would not be decidable!

The third meta-property that could be useful when evaluating expressiveness relations (except, naturally, when dealing with formula-logics) is the deduction theorem. Though involved in many formulation issues, as we shall see, a logic has a deduction theorem when it has the capacity to express in the object language its deductibility relation. Thus, other things being equal, a logic having this capability is intuitively more expressive than another one lacking it. Therefore, it is desirable that an expressiveness relation carries with it the deduction theorem, so that (a) below apparently should hold 

\vspace{1mm}
\begin{minipage}{0.1\textwidth}
(a)
\end{minipage}
\begin{minipage}{0.9\textwidth}
if $\L_2$ is more expressive than $\L_1$, and $\L_1$ has a deduction\\ theorem, then so does $\L_2$. 
\end{minipage}\vspace{2mm}

We have some issues here. Being formulation sensitive, it is complicated to define in which circumstances the existence of a deduction theorem for a logic implies its existence in another logic, whenever there is an expressiveness relation between them.  For example, a less expressive logic might have the standard deduction theorem,\footnote{To be defined below.} while the more expressive logic has only a general version of it, or perhaps lacks it completely.  This happens with Mendelson's $\FOL$,\footnote{A Hilbert-style first-order calculus with the generalization rule ``from $\phi$ infer $\forall x \phi$''. For more, see \cite[p. 76]{Mendelson-IML}.} the propositional fragment of it still satisfies the standard deduction theorem, though it fails for quantified formulas. So, it does not seem reasonable to say that this formulation of $\FOL$ is not more expressive than $\CPL$, because it does not satisfy the standard deduction theorem, since the fragment of $\FOL$ as expressive as $\CPL$ satisfies it.\footnote{The same considerations apply to $\L_{TK}$ described in \cite{Feitosa-Nascimento-Gracio-LTK} and \cite{Moreira-PHD}.} 

Cases like these constrain us to limit the role of the deduction theorem in expressiveness relations, admitting wider formulations of it.  Thus we are forced to adapt (a) accordingly so as to be able to take into account such phenomena. Finally, we have the meta-property related adequacy criterion.

\begin{quote}[\textbf{Adequacy Criterion 2}] It cannot hold that $\L_2$ be more expressive than $\L_1$ when 

\begin{itemize} 
	\item $\L_1$ is non trivial and $\L_2$ is trivial;
	\item $\L_1$ is undecidable and $\L_2$ is decidable;
	\item $\L_1$ satisfies the standard deduction theorem and the language fragment of $\L_2$ purportedly as expressive as $\L_1$ does not satisfy (not even) the general  deduction theorem;	
\end{itemize}
\end{quote}

The last criterion reflects the intuition that expressiveness is a transitive relation and there are logics that are more expressive than others.

\begin{quote}[\textbf{Adequacy Criterion 3}] (Taken from \cite{Kuijer-PHD}) The expressiveness relation should be a non-trivial pre-order, that is, it should be a transitive and reflexive relation, and there must be some pair of logics $\L_1$ and $\L_2$ such that $\L_2$ is not at least as expressive as $\L_1$.
\end{quote}

We now analyse with greater detail the criteria 1 and 2.

\subsubsection{Criterion 1- on ``whatever can be said in terms of the connectives''}\label{definability-of-connectives}

We can understand this criterion as saying ``every connective of $\L_1$ is definable in $\L_2$''. But the usual notion of definability is either treated within the same logic, or between different logics within the same class of structures.  As we intend to deal with translations between logics, the usual notion of definability is too rigid. We must give a broader reading of the criterion 1 in order to understand it as imposing an intuitive restriction on translations between logics. Thus the idea is to impose restrictions $P_1,P_2,...$ on translations so that 

\begin{quote}
$\T: \L_1 \longrightarrow \L_2$  satisfies $P_1,P_2,...$  only if, intuitively, everything that can be said in terms of the connectives of $\L_1$ can also be said in terms of the connectives of $\L_2$; let us say in shorter terms that this happens only if the connectives of $\L_1$ are \emph{generally preserved} in $\L_2$.  
\end{quote} 
 
 In the sequence some candidates for such $P_1,P_2,...$ are listed, the back-and-forth condition was given before.

\begin{definition}[Compositional]
A translation $\T: \L_1 \longrightarrow \L_2$ is compositional whenever for every $n$-ary connective $\#$ of $\L_1$ there is an $\L_2$-formula $\psi^{\#}$ such that $\T(\#(\phi_1,...,\phi_n))= \psi^{\#}(\T(\phi_1),...,\T(\phi_n))$.
\end{definition}

\begin{definition}[Grammatical]
A grammatical translation $\T$ is a back-and-forth compositional translation  such that, for a sentence $\phi$, $\T(\phi)$ may contain no other formulas other than the ones appearing in $\T(p)$, where $p$ appears in $\phi$ (thus, no parameters are allowed). 
\end{definition}

\begin{definition}[Definitional]
A definitional translation $\T$ is a grammatical translation for which $\T(p)=p$ for every atomic $p$.
\end{definition}

We have four proposals for filling the above list of restrictions. All of them require basically two conditions, taking as $P_1$ the back-and-forth condition.  In decreasing order of strictness, there is divergence in taking $P_2$ as a 

\begin{enumerate}[noitemsep]
	\item definitional translation (Wójcicki and Humberstone),
	\item grammatical translation (Epstein and apparently Koslow),
	\item general-recursive translation (to be defined below),
	\item surjective conservative translation (Mossakowski et al.).
\end{enumerate}

	Humberstone \cite{Humberstone-BTP}, recalling Wójciki's definitional translations and intuitions about expressiveness, guessed that if there is a definitional translation between $\L_1$ and $\L_2$, then all connectives in $\L_1$ are preserved in $\L_2$.\footnote{However, it seems that in \cite[p. 147]{Humberstone-BTP} he allows that connectives are preserved in a weaker way, through compositional translations.} For us, the existence of a definitional translation from $\L_1$ to $\L_2$ is the strongest guarantee that the connectives of $\L_1$ are generally preserved in $\L_2$. Nevertheless, it is too strict a requirement, and there are weaker forms of translations that can also do the job.

	For Epstein \cite[p. 302]{Epstein-SFL}, a grammatical translation is a homomorphism between languages and thus it yields a translation of the connectives. The justification is that such translations  are only possible when for each connective in the source logic, there corresponds a specific structure in the target logic that behaves similarly. Thus, through a grammatical translation, the connectives of the source logic are generally preserved in the target logic. Koslow \cite[p. 48]{Koslow-ID} also allows that a connective from one logic $\L_1$ ``persists'' in $\L_2$ if there is a homomorphism from $\L_1$ to $\L_2$.


According to Mossakowski et al. \cite{Mossakowski-WILT}, grammatical translations are too demanding for the task, as many useful and important translations are non\--gramma\-ti\-cal (e.g. the standard modal translation). For them, instead of seeking to preserve the structure of the formulas, it would be better to preserve the proof-theoretic behaviour of the connectives and to treat the connectives only as regards this behaviour (\ibid, p. 100). In this paper, some proof-theoretic conditions on the connectives are listed, e.g. for conjuntction the condition is $\Gamma \vdash \phi \land \psi$ iff $\Gamma \vdash \phi$ and $\Gamma \vdash \psi$. This formulation may lead one to think that $\land$ here shall be a logical constant, and not possibly a formula $\gamma(\phi,\psi)$ (think of $\CPL^{\restr\{\lnot,\lor\}}$, where $\gamma(\phi,\psi)= ¬(¬\phi \lor ¬\psi)$); naturally in the first case, the whole proposal would make no sense. In table \ref{Table-ptc-n}  we reformulate the conditions to reflect their proposal more clearly, where $\delta^\#$ is an arbitrary formula that stands for the connective $\#$.

\begin{table}[h]
\begin{tabular}{l l}
falsum		& \s	$\delta^{\bot}(\xi) \vdash \phi$, for every $\phi$ \\

conjunction     & \s     $\Gamma \vdash \delta^{\land}(\phi, \psi)$ \hspace{20mm} iff \s $\Gamma \vdash \phi$ and $\Gamma \vdash \psi$\\
disjunction     & \s     $\delta^{\lor}(\phi,\psi), \Gamma \vdash \chi$ \hspace{16mm} iff \s $\phi,\Gamma \vdash \chi$ and $\psi,\Gamma \vdash \chi$\\
implication     & \s     $\Gamma \vdash \delta^{\rightarrow}(\phi, \psi)$  \hspace{18.5mm} iff \s  $\Gamma,\phi \vdash \psi$ \\
negation        & \s	$\Gamma, \phi \vdash \delta^{\bot}(\xi)$ \hspace{20mm} iff \s $\Gamma \vdash \delta^\lnot(\phi)$.
\end{tabular}

\caption{Reformulation of proof theoretic connectives as given by Mossakowski et al.}
\label{Table-ptc-n}

\end{table}

\begin{definition}[presence of a proof-theoretic connective]
	A proof-theoretic connective is \emph{present} in a logic if it is possible to define the corresponding operations on sentences satisfying the conditions given in table \ref{Table-ptc-n}. 
\end{definition}

We shall now investigate this idea in detail and argue that, as it is, the preservation of connectives would require mappings stricter than conservative translations otherwise the notion of the ``presence'' of a connective must be relaxed.

\paragraph{Drawbacks on the preservation of proof-theoretic connectives}

A translation $\T:\L_1 \longrightarrow \L_2$ \emph{transports} a given $\L_1$-connective $\#$ if its presence in $\L_1$ implies its presence in $\L_2$, the converse implication is called \emph{reflection} \cite[p. 100]{Mossakowski-WILT}.  It is claimed (\ibid) that if a mapping $\T: \L_1 \longrightarrow \L_2$ is conservative and surjective, then all proof theoretic connectives of $\L_1$ are transported to $\L_2$ and all proof-theoretic connectives present in $\L_2$ are reflected in $\L_1$. However, this claim must be taken with a grain of salt, let us see why.

Let $\L_1$ be a logic having a proof-theoretic conjunction according with the table \ref{Table-ptc-n} above and suppose there is a surjective conservative mapping $\T: \L_1 \longrightarrow \L_2$. For $\L_2$-formulas $\delta_1,\delta_2$, let $\Gamma \cup \{\phi,\psi\}$ be a set of $\L_1$-formulas with $\T(\phi)=\delta_1$ and $\T(\psi)=\delta_2$. Then it holds that

\begin{quote}

	($\T(\Gamma) \vdash_{\L_2} \T(\phi)$ and $\T(\Gamma) \vdash_{\L_2} \T(\psi)$) $\hspace{1mm}$ iff $\hspace{1mm}$ $\Gamma \vdash_{\L_1} \delta^\land(\phi,\psi)$ $\hspace{1mm}$ iff $\hspace{1mm}$ $\T(\Gamma) \vdash_{\L_2} \T(\delta^\land(\phi,\psi))$. 

\end{quote}

Thus, $\L_2$ would have proof-theoretic conjunction. The grain of salt is that, once no structural restriction is imposed upon $\T$, it is not necessary that  $\T(\delta^\land(\phi,\psi))$ be constructed out of $\T(\phi)$ and $\T(\psi)$. In this case, it seems at least unnatural to say that  $\T(\delta^\land(\phi,\psi))$  is an \emph{operation} on the sentences $\T(\phi)$ and $\T(\psi)$. 

Therefore,  we must relax what it means for a connective to be \emph{present} in a logic. One has to say that e.g. the proof-theoretic conjunction is present in a logic $\L_2$ if, for all formulas $\delta_1,\delta_2$ and set of formulas $\Delta$, there is a formula $\gamma$ such that ($\Delta \vdash_{\L_2} \delta_1$ and $\Delta \vdash_{\L_2} \delta_2$) iff $\Delta \vdash_{\L_2} \gamma$.  A similar reformulation should be given for the other connectives. In this case, though, whenever it holds that $\Delta \vdash_{\L_2} \delta_1$ and $\Delta \vdash_{\L_2} \delta_2$, then any $\L_2$-theorem in the place of $\gamma$ serves to satisfy this condition for conjunction. 

For example, take a Tarskian logic $\L$ defined on the signature $\{p,q,r, \top\}$, where $p,q,r$ are propositional variables and $\top$ the constant for logical truth. Then $\L$ has proof-theoretic conjunction since $p,q \vdash p$  and $p,q \vdash q$ holds iff $p,q \vdash \top$. This is probably unproblematic and a consequence of the meaning of $\top$.  Nevertheless, for some cases  this approach to the presence of connectives has some downsides. For example, restrict $\L$ to the signature $\{p,\top\}$. Then $\L$ has the proof-theoretic conditional, since it holds that
\begin{itemize}
\item[] $p \vdash p$ iff $\vdash \top$, \s $\top \vdash p$ iff $\vdash p$, \s $p \vdash \top$ iff $\vdash \top$ \hspace{2mm} and \hspace{2mm}  $\top \vdash \top$ iff $\vdash \top$. 
\end{itemize}
But if the signature were incremented by another variable $q$, then the resulting system would no longer have a proof-theoretic conditional, since for no $\delta$ it would hold that $p \vdash q$ iff $\vdash \delta$. This volatility of the presence of proof-theoretic connectives is unreasonable.

Recapitulating, the idea of this approach is that one shall define the mappings so as to preserve the proof-theoretic connectives, instead of requiring the mappings themselves to preserve the structure of the formulas.  But if the mappings do not respect the structure of the formulas, what shall be called the presence of a connective, must also be relaxed. 

 Besides the inconvenients mentioned above, this proposal would be too restrictive in some cases. For example, Statman's translation \cite{Statman-IPLPSC} of  $\IPL$ into its implicational fragment shows  how can one ``express'' (in some sense of the term) conjunctions using only implicational formulas; recent works have generalized this result so that any logic having a certain natural deduction formulation and having the sub-formula principle is translatable into the implicational fragment of minimal logic \cite{Hermann-PLCSP}.\footnote{The idea of these translations is the following: for a given $\IPL$-formula $\phi$, take all sub-formulas $\delta_1,\delta_2$ and associate to it implicational axioms of the sort $x_{\delta_1 \land \delta_2} \rightarrow x_{\delta_1}$,  $x_{\delta_1 \land \delta_2} \rightarrow x_{\delta_2}$ and $x_{\delta_1} \rightarrow (x_{\delta_2} \rightarrow x_{\delta_1 \land \delta_2})$, where $x_{\delta_1}$, $x_{\delta_2}$ and $x_{\delta_1 \land \delta_2}$ are fresh variables.}  Nevertheless, not even in the weaker sense given above the conjunctions are ``present'' in $\IPL^{\restr \{\rightarrow\}}$.

Anyway, it must be borne in mind that to give a good and general definition of when a connective or operator is generally preserved is a difficult and spinous topic. Below we give another proposal, which is at the same time weaker (the translation mentioned above would enter) and stronger (requires structure-attentive mappings).

	Let us now consider the structure-attentive translations and think on the minimum conditions on the preservation of the structure of formulas that would allow for a reasonable and general notion of preservation of connectives.

\paragraph{General-recursive translations: allowing context-sensitivity in a general preservation of connectives}
\label{general-recursive}

	The criterion of compositionality given above a priori seems a reasonable condition for the preservation of connectives through translations. Notice that in the criterion the function $\T$ that translates $\#(\phi_1,...,\phi_n)$ is the same that translates the sub-formulas $\phi_i$. From this comes the compositionality: a translation $\T$ of a formula is obtained through the same translation $\T$ of its sub-formulas.

	Thinking about the issue of translating a connective, it is also reasonable that the translation be sensitive to the context where the connective is inserted.  This is the case in the translation $(\T_+): Grz \longrightarrow S4$ in \cite{Demri-Gore-2000}, where  $\T_+(¬\square p)= ¬\square p$, but $\T_+(\square p) = \square (\square(p \rightarrow \square p) \rightarrow p)$. Therefore, $\T_+$ distinguishes between translating $\square$-formula and $¬\square$-formula, and this is done through the help of an auxiliary translation (see complete definition in section \ref{examples}).  Thus, there are translations between some logics where the mappings must be context-sensitive, so as to convey the proper meaning of some source connectives in the target logic. 

	There are also those cases where the connectives can be dealt context\--in\-de\-pen\-dently but auxiliary translations are needed anyway. The standard translation of modal logic to $\FOL$, besides some parameters, needs $n$ auxiliary translations for each formula of modal degree $n$ e.g.  as $\T^{x}(p)=Px$ but $\T^{x}(\square \phi)= \forall y (Rxy \rightarrow \T^{y}(\phi))$.

For the sake of simplicity, we will restrict our notion of context-sensitivity to whether or not the connective to be translated is in the scope of an unary operator. When the translation of a $n$-ary connective $\#$ is sensitive as to whether it is on the scope of an unary $\circ$, a simple solution is to treat $\circ\#$ as a composite $n$-ary connective to be translated. With the aim of capturing these cases, let us consider  a sufficiently general kind of translation.

	French in \cite{French-PHD} presents a concept of recursively interdependent translation that includes non-compositional translations that are still defined recursively through the formation of formulas. A generalization of his concept will be employed here, since the original has an unmotivated restriction allowing only unary auxiliary mappings. The generalization allows auxiliary mappings of any arity and also has a simpler notation. Let $\L_1=(\F_1, \vdash_{\L_1})$ and $\L_2=(\F_2,\vdash_{\L_2})$ be logics,

\begin{definition}[General-Recursive] 
 Let $\T'_{1},...,\T'_{w}$ be auxiliary mappings of any arity defined inductively on $\F_1$-formulas.  A translation $\T: \F_1 \longrightarrow \F_2$ from $\L_1$ to $\L_2$ is \emph{general-recursive} if, for every $n$-ary connective $\#$ and formulas $\phi_1,...,\phi_n \in \F_1$, there is an $\L_2$-formula $\#^\T(p_1,...,p_m)$ \textbf{containing only} the shown propositional variables $p_1,...,p_m$, such that 
\begin{quote}
	$\T(\#(\phi_1,...,\phi_n)) = \#^\T(\T'_{1}(\phi_i,...,\phi_j)/p_1,...,\T'_{w}(\phi_h,...,\phi_l)/p_{m})$
\end{quote}
where $\{\phi_i,...,\phi_j\} \cup \{\phi_h,...,\phi_l\} \subseteq \{\phi_1,...,\phi_n\}$.
\end{definition}

	Notice that the clauses must be given for each single connective in the source logic. If there is a need to translate a composite connective, an additional clause for it should be given.

Therefore, the general-recursive translations are still structure-preserving and must be defined inductively through the formation of formulas. Later in section \ref{preserv-conn}  we argue that, together with some other conditions, general-recursive translations preserve, in a general but reasonable sense of the term, the connectives of the source in the target logic.

\paragraph*{Another issue with translated connectives}

One might insist whether the behaviour of the defined connective in the target logic would indeed be equivalent with the behavior of the original connective. Corcoran argues that this is often not the case. In \cite[p. 172]{Corcoran-TLT} he defines a notion of ``deductive strength'' which is based on the capacity of a logic to introduce and eliminate a connective occurring as a principal sign in a formula. 

Considering this notion, it can be that a given connective $\#$ of a logic $\L_1$ be definable in a logic $\L_2$ through a translation, nevertheless, the ``deductive strength'' of $\L_2$ as regards $\#$ is lower than the corresponding one in $\L_1$. For example, consider the two classical propositional logics $\CPL^{\restr\{¬,\rightarrow\}}$ and $\CPL^{\restr\{¬,\land\}}$, formulated as natural deduction systems. Translating the conditional from $\CPL^{\restr\{¬,\rightarrow\}}$ to $\CPL^{\restr\{¬, \land\}}$ one obtains the following rule of inference: from the pattern of reasoning from $\phi$ to $\psi$, infer $¬(\phi \land ¬\psi)$. Contrary to the rule for $\rightarrow$ in $\L_1$ (the usual natural deduction rule), according to Corcoran this rule of $\L_2$ is not rigorous, since it depends on the rules of the other connectives. 

Corcoran's considerations are very interesting, but we think that despite the fact that the defined connectives can lose ``deductive strength'' (in his terms) it is reasonable to say that they maintain expressive strength. The loss of ``deductive strength'' can influence other issues such as modularity, normalization, etc. but it does not affect directly expressiveness.

	Having revised the literature linked with the adequacy criterion 1 and stated our proposal, now the same will be done with respect to the adequacy criterion 2.

\subsubsection{Criterion 2- On the preservation of (some) meta-properties}

\label{Adequacy-criterion-2}

There are two important issues here:

\begin{enumerate}
\item[(i)] what is being understood as a meta-property
\item[(ii)] what does it mean for a translation to preserve a meta-property of one logic into another
\end{enumerate}

It would seem desirable to have a general formal framework so that one could give precise answers to (i) and (ii). Nevertheless, the adequacy criterion 2 asks for preservation of specific meta-properties, and not of every meta-property of a certain kind. Thus, there is no need to place them in a fixed framework. Moreover the first two (non-triviality and decidability) have simple and exact formulations, so it is straightforward to stablish whether they are preserved by a translation. 

The only meta-property whose statement and definition of preservation need elucidation is the deduction theorem.  As the framework(s) of (hyper) contextual translations\footnote{See \cite{Carnielli-Coniglio-Dottaviano-NDTBL}, \cite{Coniglio-Figallo-FFHCF} and \cite{Moreira-PHD} for a detailed presentation of both.}  offers exact answers to (i) and (ii) above, we will investigate whether they are adequate for our purposes.  In both, a logic is taken as an assertion calculus containing a set of formulas  and a set of rules of inference between sequents. The difference between the frameworks is the kind of sequent allowed.  The language of the assertion calculus includes schematic variables for sentences $\xi_1,\xi_2,...$ and set of sentences $X_1,X_2,...$, so the calculus has also substitution and instantiation rules for dealing with those.

In these frameworks, $P$ is a meta-property of a logic $\L$ defined in the above terms whenever $P$ can be formulated as an inference between sequents (or hyper-sequents), that is, if $P$ can be formulated as a derived rule of $\L$.  For example, the deduction theorem can be formulated this way: from $\Gamma, \phi \vdash \psi$, infer $\Gamma \vdash \phi \rightarrow \psi$. Now for the disjunctive property, one needs the richer framework of the hyper-sequents: from $\Gamma \vdash \phi \lor \psi$, infer $\Gamma \vdash \phi$ or infer $\Gamma \vdash \psi$.

A (hyper) contextual translation $\T: \L_1 \longrightarrow \L_2$ is a mapping that is transparent to the schematic variables such that  if $P$ is a meta-property of $\L_1$ then, $\T(P)$ is a meta-property of $\L_2$. The transparency to schematic variables implies that (hyper) contextual translations by definition preserve structural properties such as left weakening: from $X \vdash \xi$, infer $X,X' \vdash \xi$.

Consider now the finiteness property, i.e. if $\Gamma \vdash \phi$, then for a finite $\Delta \subseteq \Gamma$, $\Delta \vdash \phi$. It cannot be formulated in neither of the cited frameworks, indeed the same holds for the majority of other relevant meta-properties of logics: decidability, interpolation, cut-elimination, etc.

Despite the limitation on expressible meta-properties, the framework of (hyper) contextual translations has also a clear answer to item (ii) above: a translation $\T$ preserves a meta-property P of the source logic  if $\T(P)$ is a derived rule of the target logic.

In the sequence we use the example of the deduction theorem to argue that there can be some problems even with this strict notion of meta-property preservation.

\paragraph{A limitation of the formulation of meta-property in the framework of (hyper) contextual translations}

Even in the strict framework of (hyper) contextual translations, meta-properties are formulation-sensitive, so that one formulation of a meta-property $P$ may hold for a logic $\L$ while other formulation $P'$ fails for $\L$. A paradigmatic example is the deduction theorem. In most formulations, e.g. for classical propositional logic ($\CPL$) and intuitionistic propositional logic ($\IPL$), it is read as: If $\Gamma, \phi \vdash \psi$, then $\Gamma \vdash \phi \rightarrow \psi$. Nevertheless, only a generalized version holds for Lukasiewicz $\L^3$: If $\Gamma, \phi \vdash \psi$, then $\Gamma \vdash \phi \rightarrow (\phi \rightarrow \psi)$ \cite{Pogorzelski-DTLMVPC}. 

A similar issue occurs in systems containing proof-rules besides inference-rules,\footnote{A proof rule is of the form: from $\vdash \phi$, infer $\vdash \psi$. An inference rule is of the form: from $\phi$, infer $\psi$. The necessitation rule and generalization rules are sometimes defined as proof-rules: from $\vdash \phi$, infer $\vdash \square \phi$; from $\vdash \psi(x)$ infer $\vdash \forall x \psi(x)$. Notice that both $\phi$ and $\psi(x)$ must be theorems in their respective systems, otherwise one gets implausible inferences: from $p$ it follows $\square p$, and from $P(x)$ it follows that $\forall x P(x)$.} for example, Mendelson's $\FOL$ and modal logic with the necessitation rule. In both cases, only a modified version of the deduction theorem holds. For modal logic $K$, among other possibilities, the following deduction theorem holds \cite[p. 58]{Zeman-DT}: if $\Gamma, \phi \vdash_K \psi$, and each of the propositional variables appearing in hypothesis $\Gamma \cup \{\phi\}$ is in the scope of a modal operator, then $\Gamma \vdash_K \phi \rightarrow \psi$.\footnote{Other formulation is given in \cite{Hakli-Negri-DDTFML}: if $\Gamma, \phi \vdash_K \psi$, and the rule of necessitation is applied $m \geq 0$ times to formulas that depend on $\phi$, then $\Gamma \vdash_K (\square^0 \phi \land ... \land \square^m \phi) \rightarrow \psi$, where $\square^0 \phi = \phi$, $\square^1 \phi= \square \phi$, etc.}

For Mendelson's system, the formulation of the deduction theorem is also clumsy \cite[p. 80]{Mendelson-IML}: ``Assume that, in some deduction showing that $\Gamma, \phi \vdash \psi$, no application of [the generalization rule] to a wff that depends upon $\phi$ has as its quantified variable a free variable of $\phi$. Then, $\Gamma \vdash \phi \rightarrow \psi$.''

So what is a deduction theorem? According to Zeman, the general statement of it might be \cite[p. 56]{Zeman-DT} 

\begin{definition}[DT] If there is a proof from the hypotheses $\phi_1,...,\phi_n$ for the formula $\psi$, then there is a proof from the hypotheses $\phi_1,...,\phi_{n-1}$ for the formula $\phi_n \supset \psi$.
\end{definition}

For Zeman, the problem of the formulation of the deduction theorem for each system lies in the proper understanding in the system of what it is meant by a ``proof from hypotheses''. Thus, the different results cited above for $\L^3$, modal logic and Mendelson's $\FOL$ are different ways ---seemingly equivalent modulo the specificities of each system--- of capturing the idea of DT above.

The situation is explained by Hakli and Negri \cite{Hakli-Negri-DDTFML} as follows. For \emph{some} logics, either one modifies their rules in order for them to deal adequately with assumptions, and get the ``standard formulation'' of the deduction theorem, or leave the rules from the logic intact and obtain a ``non-standard'' form of the deduction theorem.\footnote{Although the moral of the story holds for first-order logic, as shown above, they only mentioned modal logics. Nevertheless, we do not know of any such rectification for the formulation of the deduction theorem in Lukasiewicz $\L^3$.}

Now let us come back to the issue of preservation of meta-properties by translations. Consider $\IPL$ and modal logic $S4$,  presented in the framework of (hyper) contextual translations, e.g. both equipped with a common set of propositional variables $p_1,p_2,...$ and schematic variables  $\xi_1,\xi_2,..., X_1,X_2,...$, for formulas and sets of formulas, respectively, etc.  Consider Gödel's translation $\T^g:\IPL \longrightarrow S4$ (defined to be literal to schematic variables):

\begin{multicols}{2}
\begin{description}[noitemsep]
\item $\T^g(p_i) = \square p_i$
\item $\T^g(X_i) = X_i$
\item $\T^g(\xi_i) = \xi_i$
\item $\T^g(¬\phi) = \square ¬\T^g(\phi)$
\item $\T^g(\phi \rightarrow \psi)= \square(\T^g(\phi) \rightarrow \T^g(\psi))$
\item literal for $\bot,\land,\lor$.
\end{description}
\end{multicols}

Then the deduction theorem for $\IPL$ is defined as the following meta-property 

\vspace{1mm}
\begin{minipage}{0.2\textwidth}
$(P)$
\end{minipage}
\begin{minipage}{0.7\textwidth}
if $X, \xi_1 \vdash \xi_2$, then $X \vdash \xi_1 \rightarrow \xi_2$.
\end{minipage}\vspace{1mm}

Carnielli et al. (\cite[p. 13]{Carnielli-Coniglio-Dottaviano-NDTBL}) notice that $\T^g$ above is not a contextual translation since $S4$ does not satisfy\vspace{1mm}

\begin{minipage}{0.2\textwidth} 
$\T^g(P)$
\end{minipage}
\begin{minipage}{0.7\textwidth}
if $X, \xi_1 \vdash \xi_2$, then $X \vdash \square(\xi_1 \rightarrow \xi_2)$. 
\end{minipage}\vspace{2mm}

To see why, instantiate $X$ to $p_1 \rightarrow p_2$ and $\xi_i$ to $p_i$, for $i \in \{1,2\}$. In $S4$ it holds that $p_1 \rightarrow p_2, p_1 \vdash p_2$, but it does not hold that $p_1 \rightarrow p_2 \vdash \square(p_1 \rightarrow p_2)$.

One can see clearly that this is caused by the transparency given in $\T^g$ to the schematic variables of $P$. If $P$ were formulated in terms of non-schematic formulas, e.g.\vspace{1mm}

\begin{minipage}{0.2\textwidth} 
$(P')$
\end{minipage}
\begin{minipage}{0.7\textwidth}
if $\Gamma, p_1 \vdash p_2$, then $\Gamma \vdash p_1 \rightarrow p_2$,
\end{minipage}\vspace{2mm}

then its translation $\T^g(P')$ into $S4$ would be \vspace{2mm}

\begin{minipage}{0.2\textwidth} 
$\T^g(P')$
\end{minipage}
\begin{minipage}{0.7\textwidth}
if $\T^g[\Gamma],\square p_1 \vdash \square p_2$, then $\T^g[\Gamma] \vdash \square(\square p_1 \rightarrow \square p_2)$,
\end{minipage}\vspace{2mm} 

which is satisfied in $S4$.

Although $P$ is a correct formulation of DT (see above) for $\IPL$, $\T^g(P)$ is not the correct formulation of DT for $S4$. Therefore, the claim that contextual and hyper-contextual translations preserve the meta-properties of logics (expressible in the framework) is not entirely justified. The opacity given for the schematic variables may give a``false negative'' as regards the presence of some meta-property in the target logic, this would prevent the definition of such translations. Therefore this framework is not adequate for our purposes.

\paragraph{General statement and preservation of the deduction theorem}

Recall our discussion on the deduction theorem. We saw that there are many formulations of it, and it depends on how the notion of proof from assumptions is treated in each logic. Now to talk about the preservation of deduction theorem through the translations, we have to give a sufficiently general formulation of it, but such that it still carries the spirit of Zeman's definition. Let us give it a more direct formulation:

\begin{definition}[\textsc{standard} deduction theorem]

A logic $\L_1$ has the \textsc{standard} deduction theorem whenever it holds that $\phi_1,...,\phi_n \vdash_{\L_1} \psi$ if and only if $\phi_1,...,\phi_{n-1} \vdash_{\L_1} \phi_n \rightarrow \psi$.
\end{definition}

The general formulation has to be lax enough so as to enable one to say that, for example, the translation $\T^{l}:\CPL \longrightarrow \L^3$ preserves the deduction theorem, since it holds that ``if $\Gamma, \phi \vdash_{\L^3} \psi$, then $\Gamma \vdash_{\L^3} \phi \rightarrow (\phi \rightarrow \psi)$''; analogously for the translation of $\IPL$ into $S4$. The general version of the deduction theorem we propose is the following:

\begin{definition}[\textsc{general} deduction theorem]

	A logic $\L_1$ has the \textsc{general} deduction theorem whenever $\phi_1,...,\phi_n \vdash_{\L_1} \psi$ iff $\phi_1,...,\phi_{n-1} \vdash_{\L_1} \alpha^{\rightarrow}(\phi_n,\psi)$, where $\alpha^{\rightarrow}$ is an $\L_1$-formula, with one or more occurrences of $\phi_n$ and $\psi$. 

\end{definition}

In abstract algebraic logic this formulation is known as the uniterm global deduction-detachment theorem \cite[p. 36]{Font-Jansana-Pigozzi-SAAL}.

\begin{definition}[preservation of the general deduction theorem]

	A translation $\T:\L_1 \longrightarrow \L_2$ is said to preserve the general deduction theorem whenever $\L_1$ has the standard deduction theorem and $\T(\L_1)$ has the general deduction theorem.  
\end{definition}
The case where $\L_1$ satisfies only the general deduction theorem is more complex, as it will be seen below.

\subsection{$expressiveness_{gg}$: a sufficient condition for expressiveness}

In the adequacy criteria we proposed for expressiveness, there appears two informal necessary conditions: preserving the connectives and behaving in the appropriate way as regards the selected meta-properties. The other condition of being a non-trivial pre-order is already given precisely. The first two conditions are open to interpretation, so we proposed a precise formulation of the minimal requirements such interpretations would have to satisfy. This amounted on requiring the translation to preserve the general deduction theorem, and to be back-and-forth general-recursive.

The one-way mappings between logics are a very weak in the sense they are almost omnipresent, so that requiring back-and-forth mappings as a formal necessary condition for  expressiveness is rather uncontroversial. Nevertheless, requiring structure-attentive translations in order to preserve the connectives has been questioned by Mossakowski et al. as we have seen above. They proposed other way to preserve the connectives without requiring such translations. We think this approach has some downsides and proposed a different one, based on general-recursive translations.

Being a general-recursive mapping is a relatively weak condition on translations. If some translation does not comply with it is because at least some connective of the source logic is only translated ``globally", i.e. a formula containing it is translated as a whole, and the translation ignores its eventual sub-formulas, e.g. Glivenko's double negation translation of $\CPL$ into $\IPL$ \cite{Glivenko-SQPLB}.

Thus, it is reasonable to require general-recursiveness as a formal necessary condition for expressiveness, along with the  back-and-forth condition. One very important issue to be dealt with in a future work is already pointed out by Mossakowski et al. (\cite{Mossakowski-WILT}): this minimal notion of structure preservation is up to now only defined for propositional logics. It is to be investigated how it should deal with quantifiers. Notice, however, that this limitation does not weaken the necessary character of general-recursive translations, as an eventual wider approach should include it.

Before we present a sufficient formal criterion for our concept, whose content surely is no surprise by now, some adjustments must be made concerning the preservation of the general deduction theorem. The general statement of the deduction theorem still involves compositionality: the formula $\alpha^{\rightarrow}(\phi_n,\psi)$ at issue should have $\phi_n$ and $\psi$ as sub-formulas. In order to assure this, we have to define a slightly stricter notion of general-recursive translation, which we call general-recursive$^C$:

\begin{definition}[general-recursive$^C$] A translation $\T$ is ge\-ne\-ral\--re\-cur\-sive$^C$ iff $\T$ is general-recursive and it is compositional for the conditional symbol, that is, for a formula $\phi \rightarrow \psi$ in the source logic and a template-formula $C^\T(p_1,...,p_n)$ in the target logic, $\T(\phi \rightarrow \psi)= C^\T(\T(\phi),...,\T(\psi))$. Thus the translated formula may contain as sub-formulas one or more occurrences of $\T(\phi)$ and $\T(\psi)$.
\end{definition}

The restriction on general-recursive translations is intended to assure that for the logics satisfying the standard deduction theorem, at least the translation clause for the conditional is compositional. This will rule out clauses such as $\T(\phi \rightarrow \psi)= C^\T(\T_{i_1}(\phi),...,\T_{i_n}(\phi),...,\T_{j_1}(\psi),...,\T_{j_n}(\psi))$, for $\T_{ij}$ different from $\T$. Otherwise, it could happen that the resulting translation $C^\T$ of the conditional $\phi \rightarrow \psi$ does not contain $\T(\phi)$ and $\T(\psi)$ as sub-formulas. Then it would not be reasonable to say that such formula $C^\T$ expresses the deductibility relation between $\T(\phi)$ and $\T(\psi)$.

Now we present a sufficient criterion for expressiveness

\begin{definition}[$expressiveness_{gg}$]

	A logic $\L_2$ is at least as $expressive_{gg}$ as $\L_1$ if and only if  there is a back-and-forth general-recursive (for short, B\&F-GR) translation $\T$ from $\L_1$ to $\L_2$, such that $\T$ does not require model-mappings. If $\L_1$ satisfy the standard deduction theorem, then $\T$ must be general-recursive$^C$.

\end{definition}

Below it will be shown that $expressiveness_{gg}$ satisfies the adequacy criteria given above.

\subsubsection{Adequacy criterion 1}
\label{preserv-conn}

	As we mentioned before (section \ref{definability-of-connectives}) there is some consensus in the literature that preservation of connectives requires at least compositional back-and-forth translations. The various non-compositional translations show that we could have a wider notion of preservation of connectives. We now argue that through general-recursive translations and some other conditions, it is still guaranteed that whatever can be said in terms of the source connectives can be said in terms of the target connectives.

	\begin{proposition}[Connetive preservation (general sense)] If a translation $\T: \L_1 \longrightarrow \L_2$  is back-and-forth and general-recursive (B\&F-GR), and $\T$ does not require model translations to convey the meaning of some connective in $\L_1$, then the connectives of $\L_1$ are preserved (in a general sense) in $\L_2$.
\end{proposition}

	The back-and-forth condition shall mean either a theoremhood- or de\-ri\-va\-bi\-li\-ty\--preserving translation, depending on whether one is considering formula logics or Tarskian logics, respectively. A back-and-forth translation assures a certain similarity between the global deductive behaviour of the source and  target formulas. Though, as Je{\v r}ábek's result shows, it is not enough for any reasonable notion of connective preservation and there must be some extent of structure preservation.  In this sense, the advantage of compositional translations is that they are particularly regular, so that each connective in the source  is associated to a fixed schema in the target logic, and the translation clauses are clearer. But this can be also a limitation on the means of translation, comparable to restricting translations in ordinary language to word-to-word mappings. 

	There are many cases of logics where the translation of certain operators must consider their context, so that they have to be translated in block. We gave before some examples, for another one consider Balbani and Herzig's \cite{Balbiani-Herzig-TMLPK4} translation $\T^{bh}$ of modal provability logic G into K4. For $\T^{bh}$ the result of translating $\square \phi$ also depends on whether it occurs in the scope of a negation sign:  $\T^{bh}(¬\square p)= ¬\square p$, but $\T(\square p) = \square (\square p \rightarrow p)$. These cases cannot be captured in compositional translations and can only be dealt with in more complex non-compositional ones.

The following clause is proposed as a refinement of adequacy criterion 1, capturing more precisely when a connective or group of connectives is generally preserved by a translation from $\L_1$ to $\L_2$.
\vspace{2mm}

\hspace{-7mm}\begin{minipage}{0.1\textwidth}
($\alpha$)
\end{minipage}
\begin{minipage}{0.9\textwidth}
	for each $n$-ary (composite) connective $\otimes$ in $\L_1$ and  $\L_1$-formulas $\phi_1,...,\phi_n$, there must be $\L_2$-formulas $\delta^{\otimes}(p_1,...,p_m)$ (possibly $m \not = n$) and $\psi_1,...,\psi_m$ such that  $\otimes(\phi_1, ..., \phi_n)$ has a  similar deductive behaviour with $\delta^{\otimes}(\psi_1/p_1,...,\psi_m/p_m)$. 
\end{minipage} \vspace{2mm}

	It is easy to see that back-and-forth general-recursive (B\&F-GR) translations satisfy clause ($\alpha$). In general-recursive translations, every connective of the source logic must be given inductive translation clauses, and translations for composite connectives may be given either as extra clauses, or by means of auxiliary translations. Thus, for a $n$-ary (composite) connective $\otimes$ in $\L_1$, the formula $\otimes(\phi_1,...,\phi_n)$ must be mapped by a GR-translation to a formula $\delta^{\otimes}(\psi_1,...,\psi_m)$, where each $\psi_i$ is obtained from the translation of some $\phi_k$. Now if the translation is back-and forth, then $\otimes(\phi_1,...,\phi_n)$ will have a similar deductive behaviour with $\delta^{\otimes}(\psi_1,...,\psi_m)$, since $\Gamma \vdash_{\L_1} \otimes(\phi_1,...,\phi_n)$ iff $\T(\Gamma) \vdash_{\L_2} \delta^{\otimes}(\psi_1,...,\psi_m)$.

	Nevertheless, the satisfaction of ($\alpha$) is still not enough guarantee for general preservation of connectives, as Epstein's translation $\T^E:\R \longrightarrow \CPL$ we saw above (section \ref{problem-expressiveness-g}) is B\&F-GR.  We asserted then that the relatedness implication ``$\rightarrow$'' is not expressible in $\CPL$. The issue is that $\T^E$ uses the backdoor of the model-mapping to make the translated formula  $(p \supset q) \land d_{p,q}$ true whenever the source formula $p \rightarrow q$ is true. However, if uninterpreted, the translated formula does not have the same meaning as the original formula.  Thus, the meaning of the relatedness implication is not really expressed in terms of the connectives of $\CPL$. 
	
	Thus, the last part of the proposition above is intended to force the source connectives to be defined entirely in terms of the target connectives and not smuggled by the model-translations.\footnote{This might be seen as forcing the connectives in a certain logic to be given first an adequate set of axioms/rules of inference in order to be translatable. This would agree with Zucker maxim that the meanings of the connectives must not be imposed from the outside \cite[p. 518]{Zucker-APCL}. Nevertheless, this restriction does not prohibit non-axiomatizable model-theoretic logics to be translated into each other. For example, the identity mapping from $\L(Q_0)$ to $\L(Q_0,Q_1)$ is a perfectly reasonable translation and would comply with the criterion above.}

	Therefore, if $\T: \L_1 \longrightarrow \L_2$ is B\&F-GR, then the clause ($\alpha$) is satisfied. If besides the translations of the connectives are not aided by model-mappings, then it is reasonable to say that everything expressible in terms of the  connectives of $\L_1$ are expressible in terms of the connectives of $\L_2$.

	In the beginning of the section we cited proposals for preservation of connectives via translations by Wójcicki, Epstein and Mossakowski et al. The proposal above is much weaker than the first two. As regards the Mossakowski et al.'s, this approach is both weaker and stronger:  stronger since it requires some preservation of structure; and weaker since it does not require the preservation of the proof-theoretic connectives.\vspace{2mm}

\subsubsection{Adequacy criterion 2}

\paragraph{Deduction-theorem}

In the adequacy criteria we asked that if a logic $\L_1$ has the standard deduction theorem, and $\L_2$ is at least as expressive as $\L_1$, then the language fragment of $\L_2$ as expressive as $\L_1$ has the general deduction theorem. We formulated above the standard deduction theorem  and a general version of it. As it will be seen below, in order to guarantee the preservation of  the general deduction theorem, the source logic must have the standard deduction theorem.  We fist remark that conservative general-recursive$^C$ translations preserve the general deduction theorem:

\begin{proposition} Let $\L_1$ with conditional symbol ``$\rightarrow$" satisfy the \textsc{standard} deduction theorem. If $\T: \L_1 \longrightarrow \L_2$ is a conservative ge\-ne\-ral-re\-cur\-si\-ve$^C$ translation, then $\T(\L_1)$ has the general deduction theorem.
\end{proposition}

\begin{proof}

Let the hypotheses of the proposition be satisfied. \\
Then $\T(\phi_1)$,...,$\T(\phi_n)\vdash_{\L_2} \T(\psi)$ iff $\phi_1,...,\phi_n \vdash_{\L_1} \psi$ iff $\phi_1,...,\phi_{n-1} \vdash_{\L_1} \phi_n \rightarrow \psi$, (by the standard deduction theorem) iff $\T(\phi_1),...,\T(\phi_{n-1}) \vdash_{\L_2} \T(\phi_n \rightarrow \psi)$.

By the definition of general-recursive$^C$ translations, $\T(\phi_n \rightarrow \psi)$ is a formula containing one or more occurrences of $\T(\phi_n)$ and $\T(\psi)$.   Thus, the image of $\L_1$ under $\T$ have a general deduction theorem.

\end{proof}

To preserve the general deduction theorem the source logic must have the stronger one. If $\L_1$ only satisfies the general deduction theorem and $\T: \L_1 \longrightarrow \L_2$ is B\&F-GR, we cannot guarantee that $\T(\L_1)$ satisfies the general deduction theorem.

If $\L_1$ satisfies the general deduction theorem, then $\phi_1,...,\phi_n \vdash_{\L_1} \psi$ iff it holds that $\phi_1,...,\phi_{n-1} \vdash_{\L_1} \delta^{\rightarrow}(\phi_n, \psi)$, for some $\L_1$-formula $\delta^{\rightarrow}$.  Then we have that $\T(\phi_1),...,\T(\phi_n) \vdash_{\L_2} \T(\psi)$ iff $\T(\phi_1),...,\T(\phi_{n-1}) \vdash_{\L_2} \T(\delta^{\rightarrow})$. As $\delta^{\rightarrow}$ contains $\phi_n$ and $\psi$ as sub-formulas and $\T$ is general-recursive, $\T(\delta^{\rightarrow})$ will contain $\T_{j1}(\phi_n),...,\T_{jn}(\psi)$ as sub-formulas. If $\T_{j1},...,\T_{jn}$ are equal to $\T$, then the compositionality  of the deduction theorem is saved and $\T(\L_1)$ has also a general deduction theorem. Else, if $\T_{j1},...,\T_{jn}$ are distinct from $\T$, then the general-deduction theorem is not preserved in $\T(\L_1)$.

Therefore the present approach is limited in that the source logics to be analysed in terms of expressiveness have to be put in ``proper form'' so that they satisfy the standard deduction theorem or an even wider version of it must be defined, dropping the compositionality requirement.

\paragraph{Non-triviality}

A logic $\L$ is non-trivial if for some $\L$-formulas $\phi$ and $\psi$ it holds that $\phi \not \vdash_\L \psi$. By the adequacy criteria, a trivial logic cannot be more expressive than any logic. Thus, we have to make sure that a translation intended to induce expressiveness must \emph{reflect} triviality or, alternatively, preserve non-triviality.

That is, if $\T:\L_1 \longrightarrow \L_2$, and $\L_2$ is trivial, then $\L_1$ is trivial. 

\begin{proposition}[triviality reflection] All back-and-forth translations reflect triviality.
\end{proposition}

\paragraph{Undecidability}

We commented before that the presence of decidability could be used as an indicator of the adequacy of our definition. This is because decidability is a limitation of expressiveness of a logic. Thus the condition requires that
\begin{quote} if $\L_1$ is undecidable and $\L_2$ is decidable, then $\L_2$ is not more expressive than $\L_1$. \end{quote} 
The result below shows this condition is normally satisfied. 

\begin{proposition}[\cite{Feitosa-Dottaviano-CT}] If $\L_1$ is undecidable, then there is no computable back-and-forth translation $\T:\L_1 \longrightarrow \L_2$, where $\L_2$ is decidable.
\end{proposition}


If a logic $\L_1$ is undecidable, $\L_2$ is decidable and $\T: \L_1 \longrightarrow \L_2$ is B\&F, then it follows that $\T$ would not be computable. In this case, $\T$ apparently would not be general-recursive.

\subsubsection{Adequacy criterion 3}

\paragraph{Non-trivial pre-order}

The non-triviality part of the pre-order is already fulfilled by a proposition above: there are two logics $\L_1,\L_2$, such that  $\L_1$ is non-trivial, $\L_2$ is trivial and there is no back-and-forth translation from $\L_1$ to $\L_2$. We have to prove that back-and-forth general-recursive (recall the abbreviation B\&F-GR) translations are transitive, that they are reflexive is clear.

If $\T: \L_1 \longrightarrow \L_2$ is surjective back-and-forth and $\T': \L_2 \rightarrow \L_3$ is back-and-forth, then $\T' \circ \T: \L_1 \longrightarrow \L_3$ is a back-and-forth translation \cite{Feitosa-Dottaviano-CT}.

Now if $\T:\L_1 \longrightarrow \L_2$ and $\T':\L_2 \longrightarrow \L_3$ are both B\&F-GR, being $\T$ additionally surjective, then $\T' \circ \T$ is also B\&F-GR. To see it, let $C_1,...,C_n$ and $C'_1,...,C'_n$ be the translation clauses for $\T$ and $\T'$, respectively. By surjectivity, each $\L_2$-formula is reached by some $\L_1$-formula through the applications of $C_1,...,C_n$. Thus to obtain a general-recursive mapping, is just to combine the application of the two set of clauses. Take a $\L_1$-formula $\phi$ and obtain through $C_1,...C_n$ an $\L_2$-formula $\T(\phi)$. Now apply $C'_1,...,C'_n$ to $\T(\phi)$ to obtain an $\L_3$-formula $\T'(\T(\phi))$. This translation is B\&F, and is general-recursive$^C$, since it is obtained through the clauses $C_1,...,C_n,C'_1,...,C'_n$.

Nevertheless, the surjectiveness requirement may be difficult to comply with, if one is comparing increasingly expressive logics (through language extension), e.g. propositional logic, modal logic, first-order logic. We should find a way to guarantee that whenever $\T: \L_1 \longrightarrow \L_2$ and $\T': \L_2 \longrightarrow \L_3$ are B\&F-GR, then there is a B\&F-GR translation $\T^*$ (not necessarily $\T' \circ \T$) from $\L_1$ to $\L_3$.

Let us suppose $\T:\L_1 \longrightarrow \L_2$ is a non-surjective B\&F-GR and that $\T': \L_2 \longrightarrow \L_3$ is B\&F-GR. There is naturally a weakened version $\T^{w}$ of $\T' \circ \T$, the weaker part being in the way back of $\L_3$-formulas to $\L_1$-formulas. That is, for $\L_1$-formulas $\phi_1,\psi_1$, it holds that if $\phi_1 \vdash_{\L_1} \psi_1$, then $\T'(\T(\phi_1)) \vdash_{\L_3} \T'(\T(\psi_1))$. But the converse direction only holds partially.

In order that the converse direction hold, instead of normally taking $\L_3$-formulas $\phi_3,\psi_3$ in the range of $\T'(\L_2)$, one has to take $\L_3$-formulas in the intersected range of $\T' \circ \T$. For such $\L_3$-formulas $\phi_3,\psi_3$, with $\T' \circ \T(\phi_1) = \phi_3$ and $\T' \circ \T(\psi_1)= \psi_3$ for some $\L_1$-formulas $\phi_1,\psi_1$, it holds that if $\phi_3 \vdash_{\L_3} \psi_3$, then $\phi_1 \vdash_{\L_1} \psi_1$.

But this is exactly what we wanted. The backward direction should hold only for those $\L_3$ formulas that are linked with $\L_1$-formulas through $\L_2$-formulas. To translate $\L_1$-formulas into $\L_3$ formulas $\T^w$ takes only the $\T'$-clauses $C'_1,...C'_n$ involved in translating $\T(\L_1)$ formulas.

Thus, this weakened version of $\T' \circ \T$ will suffice for us to conclude that whenever there are B\&F-GR translations $\T:\L_1 \longrightarrow \L_2$ and $\T':\L_2 \longrightarrow \L_3$, then there is a B\&F-GR translation $\T^w:\L_1 \longrightarrow \L_3$.

There remains the question whether this $\T^w$ will preserve the general deduction theorem. This will happen whenever $\T$ is compositional for ``$\rightarrow$'' and the translation clause of $\T'$ for the formula $\T(\phi \rightarrow \psi)$ is compositional. Then $\T'(\T(\phi \rightarrow \psi))$ is an $\L_3$-formula containing $\T'(\T(\phi))$ and $\T'(\T(\psi))$ as sub-formulas, which implies that $\T^w$ preserves the general deduction theorem. 

Therefore we have that

\begin{proposition}
Back-and-forth general-recursive translations form a non-trivial pre-order on logics.
\end{proposition}

From the above propositions we can conclude that

\begin{corollary}
Every back-and-forth general-recursive translation preserves the connectives and the selected meta-properties (undecidability, non-triviality and general deduction theorem) and form a non-trivial pre-order on logics.
\end{corollary}

This implies that

\begin{corollary}
Every back-and-forth general-recursive translation not aided by model-mappings agrees with adequacy criteria 1,2 and 3.
\end{corollary}

Many well known translations intuitively giving rise to an expressiveness relation satisfy $expressiveness_{gg}$. In the sequence, we briefly present them.

\subsection{Corroborating $expressiveness_{gg}$: the structure preserving translations}
\label{examples}

For the sake of supporting our notion of translational expressiveness, there follows some translations obeying the criterion that are reasonably taken as inducing an expressiveness relation.

\begin{itemize}
	\item (Wójcicki) from $\CPL$ into $\L^3$:

	\begin{itemize}[label={}]
		\item	$\T^l(p_i)=p_i$    \hfill $\T^l(¬\phi)= \T^l(\phi) \rightarrow ¬\T^l(\phi)$
		\item	$\T^l(\phi \rightarrow \psi) = \T^l(\phi) \rightarrow (\T^l(\phi) \rightarrow \T^l(\psi))$
	\end{itemize}

	\item (Gentzen) from classical first-order logic $\mathcal{CL}$ into intuitionistic first-order logic ($\mathcal{IL}$), and also from $\CL$ to Minimal first-order logic ($\M$) \cite[p. 218]{Prawitz-Malmnas}:\footnote{$\M$ is the intuitionistic logic without the rule of \emph{ex falso quodlibet.} An interesting result due to Luiz Carlos Pereira and Herman Haeusler  \cite{Pereira-Haeusler} is that this translation maps $\CL$ to \emph{any} intermediate logic between $\M$ and $\CL$.}

	\begin{itemize}[label={}]
		\item	$\T^c(Pt_1...t_n) = ¬¬Pt_1...t_n$ \hfill $\T^c(\phi \lor \psi) = ¬(¬\T^c(\phi) \land ¬\T^c(\psi))$
		
		\item	$\T^c(\exists x \phi) = ¬\forall x ¬\T^c(\phi)$ \hfill literal for $\bot, \land, \rightarrow$ and $\forall$;
	\end{itemize}
	
	\item (Gödel) from $\IL$ to $\CL$ extended with the modal system $S4$:
	
	\begin{itemize}[label={}]
		\item	$\T^s(R_it_1...t_n) = \square R_it_1...t_n$ \hfill	$\T^s(\forall x \phi) = \square \forall x(\T^s(\phi))$
		
		\item	$\T^s(\phi \rightarrow \psi) = \square(\T^s(\phi) \rightarrow \T^s(\psi))$ \hfill literal for $\bot, \land, \lor$ and $\exists$;
		
	\end{itemize}
	
	\item (Prawitz and Malmnäs) from $\IL$ to $\M$ (for $\# \in \{\land, \lor, \rightarrow\}$):

	\begin{itemize}[label={}]
		\item	$\T^{m1}(R_it_1,...,t_n)= R_it_1,...,t_n \lor \bot$   \hfill $\T^{m1}(\forall x \phi) = \forall x(\T^{m1}(\phi) \lor \bot)$

		\item	$\T^{m1}(\phi \,\,\#\,\, \psi)= (\T^{m1}(\phi) \,\,\#\,\, \T^{m1}(\psi)) \lor \bot$ \hfill $\T^{m1}(\bot) = \bot$;

	\end{itemize}

	\item (Demri and Goré) from $Grz$ to $S4$:

	\begin{itemize}[label={}]
		
		\item $\T_+(\square \phi)= \square ( \square [\T_+(\phi) \rightarrow \square \T_-(\phi)] \rightarrow \T_+(\phi))$ \hfill $\T_-(\square \phi)=\square \T_-(\phi)$		
		
		\item $\T_+(¬\phi) = ¬\T_-(\phi)$ \hfill $\T_-(¬\phi) = ¬\T_+(\phi)$

		\item $\T_+(\phi \rightarrow \psi)= \T_-(\phi) \rightarrow \T_+(\psi)$  \hfill $\T_-(\phi \rightarrow \psi) = \T_+(\phi) \rightarrow \T_-(\psi)$

		\item $\T_+$ and $\T_-$ are literal for $\land$ and atomic formulas;
	\end{itemize}
	
	\item (Van Benthem) Standard translation from modal logic to $\FOL$:	

	\begin{itemize}[label={}]
	 	\item	$\T^x(p_i)= P_ix$ \hfill $\T^x(\lozenge \phi)= \exists y (Rxy \land \T^y(\phi))$			
		
		\item	$\T^x(\square \phi) = \forall y (Rxy \rightarrow \T^y(\phi))$ \hfill literal for $¬,\land,\lor, \rightarrow, \bot$.
	\end{itemize}

\end{itemize}

\section{Conclusions}

	The commonly used precise notions of expressiveness are defined within a framework which is based on the capacity of characterizing structures, thus they apply only to model-theoretic logics. As this framework of expressiveness is defined only with respect to logics sharing the same class of structures, it was called in this text ``single-class expressiveness''. This framework can be seen as consisting of certain formula-mappings between model-theoretic logics.

	We saw two formal criteria for expressiveness, due to García-Matos and Väänänen and Kuijer, constructed in a wider framework which we called ``multi-class expressiveness''. This wider framework encompasses besides for\-mu\-la\--map\-pings, also model-mappings.  We argued that both criteria are inadequate for multi-class expressiveness. Then it was defended that moving to an even broader framework might be more promising, this is because the possibility of using model-mappings, as it happens with the counter-examples presented, opens a backdoor for ``undesirable'' translations. In the broader framework, which we called ``translational expressiveness'', a criterion for expressiveness would lack semantic notions and be based exclusively in terms of the existence of certain formula-mappings preserving the consequence relations of the logics at issue. 

	A proposal in this direction due to Mossakowski et al. was analysed and criticized, since it also over-generates. Studying the reasons for the over-generation, we proposed some adequacy criteria for relative expressiveness and a formal criterion of translational expressiveness satisfying them. The criterion is still limited in some aspects, as the notion of a structure-preserving translation is up to now only precisely defined with respect to propositional logics. The definition of structure-preserving translation is only intuitively extrapolated to quantifiers, that is, one would normally recognize a structure-preserving translation clause for a quantifier, though there is still no formal definition of it. Therefore, a truly broad formal criterion for expressiveness encompassing preservation of quantifiers is sill wanting, and we leave it for a future work.\footnote{As regards this limitation, it is curious to see that in Lindström's characterization of first-order logic as the most expressive logic satisfying countable compactness and the Löwenheim-Skolem theorem \cite{Lindstrom-EEL},  the notion of logic used does not have anything like a quantifier. Perhaps this is due to the wide interpretation of ``quantifier'' as a class of structures. Something like quantification is only required to prove a characterization with respect to the upward Tarski-Löwenheim-Skolem theorem \cite{Lindstrom-OCEL}. Only after sometime a proper ``quantifier property'' was required of an abstract logic extending first-order logic \cite{Barwise-AAMT}.}

\bibliographystyle{alpha}
\bibliography{bibliography.bib}

\end{document}